\documentclass[a4paper, 11pt]{article}
\usepackage{amsmath, amssymb, amsthm, latexsym}
\usepackage{xspace}
\usepackage{hyperref}

\makeatletter
\let\@wraptoccontribs\wraptoccontribs
\makeatother

\newtheorem{proposition}{Proposition}[section]
\newtheorem{theorem}[proposition]{Theorem}
\newtheorem{corollary}[proposition]{Corollary}
\newtheorem{conjecture}[proposition]{Conjecture}
\newtheorem{lemma}[proposition]{Lemma}

\newtheorem*{proposition*}{Proposition}
\newtheorem*{theorem*}{Theorem}
	
\theoremstyle{definition}
\newtheorem{definition}[proposition]{Definition}
\newtheorem{question}[proposition]{Question}

\newcommand{\set}[1]{\left\{#1\right\}}
\newcommand{\setcon}[2]{\left\{#1\ \left|\ #2\right.\right\}}

\newcommand{\abs}[1]{\left\lvert#1\right\rvert}
\newcommand{\C}{\mathbb{C}}
\newcommand{\HH}{\mathbb{H}}
\newcommand{\R}{\mathbb{R}}
\newcommand{\Z}{\mathbb{Z}}
\newcommand{\N}{\mathbb{N}}

\newcommand{\im}{\textup{im}}
\newcommand{\diam}{\textup{diam}}
\newcommand{\vol}{\textup{vol}}
\newcommand{\sep}{\textup{sep}}
\newcommand{\cut}{\textup{cut}}
\newcommand{\wir}{\textup{wir}}
\newcommand{\DL}{\textup{DL}}
\newcommand{\calG}{\mathcal{G}}

\title{Thick embeddings of graphs into symmetric spaces via coarse geometry}
\author{Benjamin Barrett and David Hume \\ with an appendix by Larry Guth and Elia Portnoy}

\date{\today}

\begin{document}

\maketitle
\begin{abstract}
We prove estimates for the optimal volume of thick embeddings of finite graphs into symmetric spaces, generalising results of Kolmogorov-Barzdin and Gromov-Guth for embeddings into Euclidean spaces. 
We distinguish two very different behaviours depending on the rank of the non-compact factor. For rank at least 2, we construct thick embeddings of $N$-vertex graphs with volume $CN\ln(1+N)$ and prove that this is optimal. For rank at most $1$ we prove lower bounds of the form $cN^a$ for some (explicit) $a>1$ which depends on the dimension of the Euclidean factor and the conformal dimension of the boundary of the non-compact factor.
The main tool is a coarse geometric analogue of a thick embedding called a coarse wiring, with the key property that the minimal volume of a thick embedding is comparable to the ``minimal volume'' of a coarse wiring for symmetric spaces of dimension at least $3$. In the appendix it is proved that for each $k\geq 3$ every bounded degree graph admits a coarse wiring into $\mathbb{R}^k$ with volume at most $CN^{1+\frac{1}{k-1}}$. As a corollary, the same upper bound holds for real hyperbolic space of dimension $k+1$ and in both cases this result is optimal.
\end{abstract}

\section{Introduction}
The focus of this paper is on thick embeddings of graphs as considered by Kolmogorov-Barzdin and Gromov-Guth \cite{KB,GG}. By a graph, we mean a pair $\Gamma=(V\Gamma,E\Gamma)$ where $V\Gamma$ is a set whose elements are called vertices, and $E\Gamma$ is a set of unordered pairs of distinct elements of $V\Gamma$. Elements of $E\Gamma$ are called edges. The \textbf{topological realisation} of a graph is the topological space obtained from a disjoint union of unit intervals indexed by $e\in E\Gamma$, whose end points we label using the two elements contained in $e$. We then identify all endpoints which are labelled by the same element of $V\Gamma$. We will use $\Gamma$ to refer to both the graph and its topological realisation.

The idea behind thick embeddings of graphs is that they are the appropriate embeddings to consider in situations where the graph models a physical object (i.e.\ vertices and edges are ``thick'' and therefore need to remain a prescribed distance apart). Two key examples are: a brain, where neurons are represented by vertices and axons by edges; and an electronic network, where components are vertices and wires are edges. We briefly summarise the relevant results from \cite{KB,GG} in the following two theorems.

\begin{theorem}\label{thm:KBGGub} Let $\Gamma$ be a finite graph with maximal degree $d$. For each $k\geq 3$, there is a topological embedding $f_k:\Gamma\to\R^k$ and a constant $C=C(d,k)$ with the following properties:
	\begin{enumerate}
		\item[$(i)$] $d_{\R^k}(f_k(x),f_k(y))\geq 1$ whenever $x,y$ are: two distinct vertices; an edge and a vertex not contained in that edge; or two disjoint edges.
		\item[$(ii)$] $\diam(f_3):=\diam(\im(f_3))\leq C|\Gamma|^{1/2}$.
		\item[$(iii)$] $\diam(f_k)\leq C|\Gamma|^{1/(k-1)}\ln(1+|\Gamma|)^4$.
	\end{enumerate}
\end{theorem}

Let $Z$ be a metric space. We say a topological embedding $g:\Gamma\to Z$ is \textbf{$\varepsilon$-thick} if it satisfies the inequality $d_{Z}(g(x),g(y))\geq \varepsilon$ whenever $x,y$ are as in condition $(i)$.

\begin{theorem}\label{thm:KBGGlb}
	Let $k\geq 3$. For every $\delta,\varepsilon>0$ and $d\in\N$ there is a constant $c>0$ such that given any finite graph $\Gamma$ with maximal degree $d$ and Cheeger constant (cf.\ Definition \ref{defn:Cheeger}) $\geq \delta$ and any $\varepsilon$-thick topological embedding $g:\Gamma\to\R^k$, we have $\diam(g)\geq c^{-1}|\Gamma|^{1/(k-1)}-c$.
\end{theorem}
When $Z$ admits a measure, we define the \textbf{volume} $\vol(g)$ of an $\varepsilon$-thick topological embedding $g:\Gamma\to Z$ to be the measure of the $1$-neighbourhood of its image\footnote{The choice of $1$-neighbourhood is arbitrary for measure spaces with controlled growth (cf.\ Definition \ref{defn:controlledgrowth}) replacing this by another positive real changes volume by at most some uniform multiplicative constant.}. From Theorem $\ref{thm:KBGGub}$ we get obvious upper bounds on the volume of $1$-thick embeddings into $\R^k$. Namely, $\vol(f_3)\leq C'|\Gamma|^{3/2}$ and $\vol(f_k)\leq C' |\Gamma|^{k/(k-1)}\ln(1+|\Gamma|)^{4k}$.\medskip

In the main paper, we prove versions of Theorems \ref{thm:KBGGub} and \ref{thm:KBGGlb} for thick embeddings into symmetric spaces. The goal of the appendix is to provide sharp upper bounds for thick embeddings into Euclidean spaces. The main result there is a complete proof of an optimal version of Theorem \ref{thm:KBGGub}(iii). Such an argument had previously been sketched by Guth.

\begin{theorem}\label{thm:KBRn} Let $d,k\in\N$ with $k\geq 3$. There is a constant $C=C(d,k)$ such that for every finite graph $\Gamma$ with maximal degree $d$, there is a $1$-thick topological embedding $f_k:\Gamma\to\R^k$ which satisfies $$\diam(f_k)\leq C|\Gamma|^{1/(k-1)} \quad \textrm{and} \quad \vol(f_k)\leq C|\Gamma|^{1+1/(k-1)}.$$
\end{theorem}

\subsection{Thick embeddings into symmetric spaces}
Our main results are analogues of Theorems \ref{thm:KBGGub} and \ref{thm:KBGGlb} for more general simply connected Riemannian symmetric spaces. Constructing graph embeddings into a range of symmetric spaces has applications for machine learning (see \cite{Lopez} and references therein).  In what follows we will assume that our symmetric spaces are simply connected and Riemannian. The \textbf{rank} of a symmetric space is the maximal dimension of an isometrically embedded Euclidean subspace. We recall that each symmetric space $X$ decomposes as a direct product of symmetric spaces $K\times\R^d\times N$ where $K$ is compact and $N$ has no non-trivial compact or Euclidean factor. In the literature, $N$ is often referred to as the \textbf{non-compact factor}. Our results show a striking contrast between the situation where the non-compact factor has rank at least $2$ and the situation where it has rank at most $1$. 

We begin with the case where the rank of $N$ is at least $2$, where we provide matching upper and lower bounds.

\begin{theorem}\label{thm:wirhighrkub} Let $X$ be a symmetric space whose non-compact factor has rank $\geq 2$ and let $d\in\N$. There are constants $\varepsilon,C>0$ which depend on $X$ and $d$ such that for any finite graph $\Gamma$ with maximal degree at most $d$, there is an $\varepsilon$-thick topological embedding of $\Gamma$ into $X$ with diameter $\leq C\ln(1+|\Gamma|)$ and volume $\leq C|\Gamma|\ln(1+|\Gamma|)$.
\end{theorem}

\begin{theorem}\label{thm:wirhighrklb} Let $X$ be a symmetric space whose non-compact factor has rank $\geq 2$ and let $d\in\N$. For any $d,\varepsilon,\delta>0$ there is a constant $c=c(d,\varepsilon,\delta)>0$ with the following property. 
	For any finite graph $\Gamma$ with maximal degree $d$ and Cheeger constant $h(\Gamma)\geq\delta$ every $\varepsilon$-thick\footnote{Unlike topological embeddings into Euclidean space, there does not seem to be an obvious way to relate the volumes of optimal topological embeddings with different thickness parameters.} topological embedding $g:\Gamma\to X$ satisfies $\vol(g)\geq c|\Gamma|\ln(1+|\Gamma|)$.
\end{theorem}

Now we turn to the case where the rank of $N$ is at most $1$. When $N$ is a real hyperbolic space, we also provide upper and lower bounds which match except in the case of the hyperbolic plane where there is a sublogarithmic gap.

\begin{theorem}\label{thm:wirehypxEucub} Let $X=\R^r\times\HH_\R^q$ where $q+r\geq 3$. Let $d\in\N$. There is a constant $C=C(X,d)$ such that for any finite graph $\Gamma$ with maximal degree at most $d$ there is a $1$-thick topological embedding of $\Gamma$ into $X$ with volume
\[
\leq C|\Gamma|^{1+1/(q+r-2)}.
\]
\end{theorem}
For $q+r\geq 4$ this follows by composing the topological embedding from Theorem \ref{thm:KBRn} with a suitable coarse embedding $\R^r\times\R^{q-1} \to \R^r\times \HH_{\R}^q$ where $\R^{q-1}$ embeds as a horosphere in $\HH_{\R}^q$. The case $q+r=3$ is new and is treated separately (cf.\ Theorem \ref{thm:thickH3ub}).

The lower bound we prove holds more generally. The rank one symmetric spaces of non-compact type are real, complex and quaternionic hyperbolic spaces of dimension at least $2$ ($\HH^q_\R$, $\HH^q_\C$ and $\HH^q_\HH$ respectively) and the Cayley plane $\HH^2_\mathbb{O}$. These spaces are all Gromov-hyperbolic, and as such they have a naturally defined boundary. The \textbf{conformal dimension} of the boundary of $\HH^q_F$ is $Q=(q+1)\dim_\R(F)-2$ \cite{Pan-89-cdim}. We will not define conformal dimension in this paper as we do not require it.

\begin{theorem}\label{thm:wirehypxEuclb} Let $X=K\times \R^r \times \HH^q_{F}$, where $K$ is compact and $q\dim_\R(F)+r\geq 3$. Let $d\in\N$. Let $Q$ be the conformal dimension of the boundary of $\HH^q_{F}$. For any $d,\varepsilon,\delta>0$ there is a constant $c=c(d,\varepsilon,\delta)>0$ with the following property. 
For any graph $\Gamma$ with maximal degree $d$ and Cheeger constant $h(\Gamma)\geq\delta$ every $\varepsilon$-thick topological embedding $g:\Gamma\to X$ has volume
\[
\geq \left\{\begin{array}{lll}
  c|\Gamma|^{1+1/r}\ln(1+|\Gamma|)^{-1/r} & \textup{if} & Q=1, \\
  c|\Gamma|^{1+1/(Q+r-1)} & \textup{if} & Q\geq 2.
\end{array}\right.
\]
\end{theorem}

This ``gap'' between the rank at most $1$ and the higher rank case is similar in flavour to the gap in the separation profiles of symmetric spaces found in \cite{HumeMackTess2}. This is no coincidence. The lower bounds on the volumes of topological embeddings found in Theorems \ref{thm:wirhighrklb} and \ref{thm:wirehypxEuclb} are inverse functions of the separation profiles of the symmetric spaces\footnote{By the separation profile of a symmetric space we mean either the $1$-Poincar\'e profile of the symmetric space as defined in \cite{HumeMackTess1} or equivalently, the separation profile as defined in \cite{BenSchTim} of any graph quasi-isometric to the symmetric space.}, and our approach to prove both of these theorems utilises separation profiles in a crucial way. In order to use separation profiles, we will reformulate the above theorems in terms of carefully chosen continuous maps (called coarse wirings) between bounded degree graphs. \medskip

We present one further result in this section, which provides upper bounds for thick embeddings into $\HH^3_\R$ and $\HH^2_\R\times\R$. The first is asymptotically optimal and the second within a sublogarithmic error (with the lower bounds provided by \ref{thm:wirehypxEuclb}) but which do not depend on the degree of the graph.

\begin{theorem}\label{thm:thickH3ub}
There are $1$-thick topological embeddings of $K_M$ (the complete graph on $M$ vertices) into $\HH^3_\R$ with diameter $\leq C\ln(1+M)$ and volume $\leq CM^2$, and into $\HH^2_\R\times\R$ with diameter $\leq CM$ and volume $\leq CM^2$, for some $C$ which does not depend on $M$.
\end{theorem}

\subsection{Coarse $k$-wirings}
\begin{definition}\label{defn:kwiring} Let $\Gamma,\Gamma'$ be graphs. A \textbf{wiring} of $\Gamma$ into $\Gamma'$ is a continuous map $f:\Gamma\to\Gamma'$ such that the image of each vertex is a vertex and the image of each edge is a walk in $\Gamma'$.

A wiring $f$ is a \textbf{coarse $k$-wiring} if
\begin{enumerate}
 \item the preimage of each vertex in $\Gamma'$ contains at most $k$ vertices in $\Gamma$; and
 \item each edge $e$ in $\Gamma'$ is contained in the image of at most $k$ edges in $\Gamma$.
\end{enumerate}
We consider the \textbf{image} of a wiring $\im(f)$ to be the subgraph of $\Gamma'$ consisting of all vertices in $f(V\Gamma)$ and all the walks which are the images of edges under $f$. The \textbf{diameter} of a wiring $\diam(f)$ is the diameter of its image (measured with respect to the shortest path metric in $\Gamma'$), the \textbf{volume} of a wiring $\vol(f)$ is the number of vertices in its image.
\end{definition}

Under mild hypotheses on the target space (cf.\ Definition \ref{defn:controlledgrowth}), we can convert a thick topological embedding into a coarse $k$-wiring.

\begin{proposition}\label{prop:coarse_wiring} Let $M$ be a Riemannian manifold with controlled growth and let $Y$ be a graph quasi-isometric to $M$, let $d\in\N$ and let $T>0$. There exist
    constants $C$ and $k$ such for every finite graph $\Gamma$ with maximal
    degree $d$ the following holds:

    If there is a $T$-thick topological embedding $\Gamma\to M$ with diameter $D$ and volume $V$
    then there is a coarse $k$-wiring of $\Gamma$ into $Y$ with diameter at most $CD$
    and volume at most $CV$.
\end{proposition}

With stronger hypotheses we are able to convert coarse $k$-wirings into thick topological embeddings.

\begin{theorem}\label{thm:thickemb}
Let $M$ be a compact Riemannian manifold of dimension $n\geq 3$, let $Y$ be a graph quasi-isometric to the universal cover $\widetilde{M}$ of $M$ and let $k,d\in\N$. There exist constants
    $C$ and $\varepsilon>0$ such that the following holds:

    If there is a coarse $k$-wiring of a finite graph $\Gamma$ with maximal degree $d$ into $Y$ with diameter $D$ and volume $V$ then there is a $\varepsilon$-thick embedding
    of $\Gamma$ into $\widetilde{M}$ with diameter at most $CD$ and volume at most $CV$.
\end{theorem}

All of the symmetric spaces we consider are universal covers of compact Riemannian manifolds. The reason for working with universal covers of compact manifolds is to use compactness to deduce that finite families of curves which are disjoint are at least a uniform positive distance apart. We then use deck transformations of the universal cover to translate these curves and preserve this uniform disjointness.

Using Proposition \ref{prop:coarse_wiring} and Theorem \ref{thm:thickemb} we can prove Theorems \ref{thm:wirhighrkub}, \ref{thm:wirhighrklb}, \ref{thm:wirehypxEucub} and \ref{thm:wirehypxEuclb} purely in terms of coarse wirings. We introduce wiring profiles in order to discuss coarse wirings between infinite graphs.

\begin{definition} Let $\Gamma$ be a finite graph and let $Y$ be a graph. We denote by $\wir^k(\Gamma\to Y)$ the minimal volume of a coarse $k$-wiring of $\Gamma$ into $Y$. If no such coarse $k$-wiring exists, we say $\wir^k(\Gamma\to Y)=+\infty$.

Let $X$ and $Y$ be graphs. The $k$-\textbf{wiring profile} of $X$ into $Y$ is the function
\[
 \wir^k_{X\to Y}(n) = \max\setcon{\wir^k(\Gamma\to Y)}{\Gamma\leq X,\ |\Gamma|\leq n}.
\]
\end{definition}

A simple example of a situation where $\wir^k(\Gamma\to Y)=+\infty$ is when $\Gamma$ has a vertex whose degree is greater than $k^2$ times the maximal degree of $Y$.

The reason for working with wiring profiles is that they have three very useful properties. Firstly, wirings between graphs can be composed and there is a natural inequality which controls the volume of the composition.

\begin{proposition}\label{prop:composition}  Let $X,Y,Z$ be graphs. Suppose $\wir^k_{X\to Y}$ and $\wir^l_{Y\to Z}$ take finite values. Then
  \begin{align*}
    \wir^{kl}_{X\to Z}(n) \leq \wir^l_{Y\to Z}\left(\wir^k_{X\to Y}(n)\right).
  \end{align*}
\end{proposition}

Secondly, for bounded degree graphs, the wiring profile of $X$ into $Y$ grows linearly whenever there is a regular map from $X$ to $Y$.

\begin{definition}\label{defn:regular} Let $X,Y$ be metric spaces and let $\kappa>0$. A map $r:X\to Y$ is $\kappa$-regular if 
\begin{enumerate}
 \item $d_Y(r(x),r(x'))\leq \kappa(1+ d_X(x,x'))$, and
 \item  the preimage of every ball of radius $1$ in $Y$ is contained in a union of at most $\kappa$ balls of radius $1$ in $X$.
\end{enumerate}
\end{definition}
Quasi-isometric and coarse embeddings between bounded degree graphs are examples of regular maps.

\begin{proposition}
  \label{prop:regular}
  Let $X$ and $Y$ be graphs with maximal degree $\Delta>0$ and let $r: X\to Y$ be a
  $\kappa$-regular map. Then there exists $k=k(\kappa,\Delta)$ such that
  \begin{align*}
    \wir^k_{X\to Y}(n) \leq \left(\kappa+\frac12\right)\Delta n.
  \end{align*}
\end{proposition}
These two propositions naturally combine to show that wiring profiles are well-behaved with respect to regular maps.

\begin{corollary}\label{cor:wirreg}
  Let $X$, $X'$, $Y$ and $Y'$ be graphs with maximal degree $\Delta$ and let $r_X: X'\to
  X$ and $r_Y: Y \to Y'$ be $\kappa$-regular maps. Then for every $k$ such
  that $\wir^k_{X\to Y}$ takes finite values there is some $l$ such that
  \begin{eqnarray}
    \wir^{l}_{X\to Y'}(n) & \leq & \left(\kappa+\frac12\right)\Delta \dot\wir^k_{X\to Y}(n).\label{eq:cwreg1} \\
    \wir^{l}_{X'\to Y'}(n)&  \leq & \left(\kappa+\frac12\right)\Delta \dot\wir^k_{X\to Y}\left(\left(\kappa+\frac12\right)\Delta n\right).\label{eq:cwreg2}
  \end{eqnarray}
\end{corollary}

The third benefit of coarse wirings is that we can find lower bounds on the wiring profile of two bounded degree graphs in terms of their separation profiles: a measure of the combinatorial connectivity of their finite subgraphs introduced in \cite{BenSchTim}. We introduce the following notation from that paper. Given two functions $f,g:\N\to\R$, we write $f\lesssim g$ if there is a constant $C$ such that $f(n)\leq Cg(Cn)+C$ holds for all $n$. We write $f\simeq g$ when $f\lesssim g$ and $g\lesssim f$.

\begin{theorem}\label{thm:wirsep} Let $X$ and $Y$ be graphs of bounded degree where $\sep_X\gtrsim n^r\ln(n)^s$ and $\sep_Y\simeq n^p\ln(n)^q$. Then, for any $k$, 
\[
wir^k_{X\to Y}(n)\gtrsim \left\{ 
\begin{array}{lll} 
 n^{r/p}\ln(n)^{(s-q)/p} & \textup{if} & p>0, \\
 \exp(n^{r/(q+1)}\ln(n)^{s/(q+1)}) & \textup{if} & p=0. 
 \end{array}\right.
 \]
\end{theorem}

The separation profiles of (graphs quasi-isometric to) symmetric spaces have all been calculated \cite{BenSchTim,HumeMackTess1,HumeMackTess2} and are all of the form $n^p\ln(n)^q$. Combining these calculations with Theorem \ref{thm:wirsep} and Theorem \ref{thm:thickemb} is sufficient to prove Theorems \ref{thm:wirhighrklb} and \ref{thm:wirehypxEuclb}. \medskip

The coarse geometric approach also has great benefits when computing upper bounds. For instance, we can deduce the upper bound on volumes of thick embeddings in Theorem \ref{thm:wirhighrkub} from the following theorem.
\begin{theorem}\label{thm:wirDL22} There is a Cayley graph $Y$ of the lamplighter group $\Z_2\wr\Z$ with the following property. For each $d\in\N$ there is some $C=C(d)$ such that for any $N$-vertex graph $\Gamma$ with maximal degree $d$, we have
\[
 \wir^{2d}(\Gamma\to Y) \leq CN\ln(1+N).
\] 
\end{theorem}
The deduction works as follows. The graph $Y$ is quasi-isometric to the Diestel-Leader graph $\DL(2,2)$ \cite{Woess-LAMPLIGHTERS_HARMONIC_FUNCTIONS}. Next, $\DL(2,2)$ quasi-isometrically embeds into any symmetric space $M$ whose non-compact factor has rank $\geq 2$ \cite[Proposition 2.8 and Theorem 3.1]{HumeMackTess2}. Choose a graph $X$ which is quasi-isometric to $M$. By Corollary \ref{cor:wirreg}, there are constants $l,C'$ which depend on $Y$ and $d$ but not $N$ such that $\wir^l(\Gamma\to X) \leq C'N\ln(1+N)$. Theorem \ref{thm:wirhighrkub} then follows from Theorem \ref{thm:thickemb} and Theorem \ref{thm:wirDL22}. \medskip

It is important to stress that the analogy between thick embeddings and coarse wirings only holds when there is a bound on the degree of the graphs and the manifold dimension of the symmetric space is at least $3$. This is evidenced by Theorem \ref{thm:thickH3ub} which holds independent of the degree of the graph, where no such result for coarse wirings is possible. On the other hand, in section \ref{sec:2dim}, we will consider coarse wirings into $\R^2$ and $\HH_\R^2$ where only planar graphs admit topological embeddings.

\begin{theorem}\label{thm:wirplanar}
Let $d\geq 3$ and let $X(d)$ be the disjoint union of all finite graphs with maximal degree $\leq d$. Let $Y$ and $Z$ be graphs which are quasi-isometric to $\R^2$ and $\HH_\R^2$ respectively. For all sufficiently large $k$, we have
\[
 \wir^{k}_{X(d)\to Y}(n) \simeq n^2 \quad \textup{and} \quad 
 \exp(n^{1/2})\lesssim \wir^{k}_{X(d)\to Z}(n) \lesssim \exp(n).
\]
\end{theorem} 
The lower bounds both follow from Theorem \ref{thm:wirsep}, since $\sep_{X(d)}(n)\simeq n$ as it contains a family of expanders of at most exponentially growing size \cite{HumSepExp}. For the upper bound we will make direct constructions. We believe that it is possible to improve the bound in the $p=0$ case of Theorem \ref{thm:wirsep} to $\exp(n^{r/q}\ln(n)^{s/q})$. This would have the consequence that $\wir^{k}_{X(d)\to Z}(n) \simeq \exp(n)$ in Theorem \ref{thm:wirplanar}.

One very natural question to consider is the dependence of $\wir^k_{X\to Y}(n)$ (up to $\simeq$) on the parameter $k$. It is clear that for $k\leq l$, $\wir^k_{X\to Y}(n)\geq \wir^l_{X\to Y}(n)$ for any pair of bounded degree graphs $X$ and $Y$, but the converse fails spectacularly \cite{Raistrick}.

\subsection*{Acknowledgements} The authors would like to thank Itai Benjamini
for suggesting the relationship between wiring problems and the separation
profile which provided the initial spark for this work, Romain Tessera for suggestions which improved the exposition, and an anonymous referee for many suggestions and observations which greatly improved the readability of the paper.

\section{Thick topological embeddings into products of real hyperbolic and Euclidean spaces}
Our goal in this section is to prove Theorems \ref{thm:wirehypxEucub} and \ref{thm:thickH3ub}, which we do by directly constructing thick topological embeddings. We start with the proof of Theorem \ref{thm:wirehypxEucub} in the case $q+r\geq 4$. We will use the upper halfspace model of real hyperbolic space $\HH^q_\R=\setcon{(x_1,\ldots,x_{q-1};x_q)}{x_i\in\R,\ x_q>0}$ equipped with the metric
\[
d_{\HH^q_\R}((x_1,\ldots,x_{q-1};x_q),(y_1,\ldots,x_{y-1};y_q))=\cosh^{-1}\left(1+ \frac{\sum_{i=1}^q (x_i-y_i)^2}{2x_qy_q}\right).
\]

\begin{proof} Define $h_0=(2(\cosh(1)-1))^{-1/2}$. Consider the map
\[
\phi_{q,r}:\R^r\times\R^{q-1}\to \R^r\times\HH^q_\R \quad \textup{given by} \quad \phi_{q,r}(\underline{x},\underline{y})=(\underline{x},(\underline{y};h_0)).
\]
\textbf{Claim:} $d(\phi_{q,r}(\underline{x},\underline{y}),\phi_{q,r}(\underline{x'},\underline{y'}))\geq 1$ whenever $\|\underline{x}-\underline{x'}\|_2\geq 1$ or $\|\underline{y}-\underline{y'}\|_2\geq 1$.
\begin{proof}[Proof of Claim]
If $\|\underline{x}-\underline{x'}\|_2\geq 1$ then this is obvious. If $\|\underline{y}-\underline{y'}\|_2\geq 1$, then
\begin{eqnarray*}
d(\phi_{q,r}(\underline{x},\underline{y}),\phi_{q,r}(\underline{x'},\underline{y'})) & \geq & d_{\HH^q_{\R}}((\underline{y};h_0),(\underline{y'};h_0)) \\ & = & \cosh^{-1}\left(1+ \frac{\|\underline{y}-\underline{y'}\|_2^2}{2h_0^2}\right) \\
& \geq & \cosh^{-1}\left(1+ \frac{1}{2h_0^2}\right)
\\
& = & \cosh^{-1}(1+ (\cosh(1)-1))=1.
\end{eqnarray*}
\end{proof}
Let $\Gamma$ be a finite graph with maximal degree $d$ and let $\psi=\sqrt2.f_{q+r-1}$ where $f_{q+r-1}$ is the $1$-thick topological embedding of $\Gamma$ into $\R^{q+r-1}$ defined in Theorem \ref{thm:KBRn}. Let us first show that $\psi\circ\phi$ is a $1$-thick embedding of $\Gamma$ into $\R^r\times\HH^q_{\R}$.

The topological embedding $\psi$ is $\sqrt{2}$-thick. If $\|(\underline{x},\underline{y})-(\underline{x'},\underline{y'})\|_2\geq \sqrt 2$, then either $\|\underline{x}-\underline{x'}\|_2\geq 1$ or $\|\underline{y}-\underline{y'}\|_2\geq 1$. Applying the claim, we see that $\psi\circ\phi$ is $1$-thick.

Finally we bound $\vol(\psi\circ\phi)$. Firstly note that if $\|(\underline{x},\underline{y})-(\underline{x'},\underline{y'})\|_2\leq 1$, then 
\begin{eqnarray*}
d(\phi_{q,r}(\underline{x},\underline{y}),\phi_{q,r}(\underline{x'},\underline{y'})) & = & \left(\|\underline{x}-\underline{x'}\|_2+d_{\HH^q_{\R}}((\underline{y};h_0),(\underline{y'};h_0))\right)^{1/2} \\ & \leq & \left(1+\cosh^{-1}\left(1+ \frac{\|\underline{y}-\underline{y'}\|_2^2}{2h_0^2}\right)\right)^{1/2} \\
& \leq & \left(1+\cosh^{-1}\left(1+ \frac{1}{2h_0^2}\right)\right)^{1/2}=\sqrt2.
\end{eqnarray*}
Now let $Y$ be a $\frac12$-separated $1$-net in $\im(\psi)$. It follows from the above equation that $\phi(Y)$ is a $\sqrt 2$-net in $\im(\psi\circ\phi)$. Denote by $\alpha,\beta$ the volumes of the balls of radius $\frac14$ and $\sqrt2+1$ in $\R^{q+r-1}$ and $\R^r\times\HH^q_\R$ respectively. We have
\[
\vol(\psi\circ\phi) \leq \beta |Y| \quad \textup{and} \quad \alpha|Y| \leq \vol(\psi).
\]
Hence, using the volume bounds from Theorem \ref{thm:KBGGub} as explained after Theorem \ref{thm:KBGGlb}, there is a constant $C$ which depends on $q,r,d$ but not $\Gamma$ such that
\begin{eqnarray*}
\vol(\psi\circ\phi) & \leq & \beta |Y| \\
& \leq & \beta\alpha^{-1} \vol(\psi) \\
& \leq & \beta\alpha^{-1}C'|\Gamma|^{1+1/(q+r-2)}\ln(1+\Gamma)^{4(q+r-1)}.
\end{eqnarray*}
\end{proof}

It remains to tackle the case $q+r=3$. We split the proof into two parts. Firstly, we build a $1$-thick topological embedding of the complete graph on $N$ vertices into $[0,N-1]^2\times[0,1]$. Then we use embeddings of $\R^2$ into $\HH_\R^3$ and $\HH^2_\R\times\R$ to construct $1$-thick topological embeddings. 

\begin{lemma}\label{fineGtoR2} Let $K_N$ denote the complete graph on $N$ vertices. There is a $1$-thick topological embedding $f:K_N\to[0,N-1]^2\times[0,1]\subset(\R^3,\|\cdot\|_\infty)$.
\end{lemma}
\begin{proof}
Enumerate the vertices of $K_N$ as $v_0,\ldots,v_{N-1}$. Now we map $v_k$ to $(k,k,0)$. We connect $(k,k,0)$ to $(l,l,0)$ using the following piecewise linear path $P_{kl}$:
\begin{equation}
 (k,k,0) \to (l,k,0) \to (l,k,1) \to (l,l,1) \to (l,l,0).
\end{equation}
Let us verify that this embedding is $1$-thick. Any two distinct vertices $v_k$ and $v_l$ are mapped at distance $|k-l|\geq 1$. Next, consider a path $P_{kl}$ and the image $(i,i,0)$ of a vertex $v_i$ with $i\neq k,l$. Since one of the first two coordinates of the path $P_{kl}$ is always either $k$ or $l$, we have
\[
 d_\infty(P_{kl},(i,i,0)) \geq \min\{|i-k|,|i-l|\} \geq 1.
\] 
Finally, consider paths $P_{ij},P_{kl}$. Let $(w,x,a)\in P_{ij}$ and $(y,z,b)\in P_{kl}$ and suppose $d((w,x,a),(y,z,b))<1$.

If $a=1$, then $b>0$, so $w=j$ and $y=k$. Since $d_\infty((w,x,a),(y,z,b))\geq |w-y|$, we have $|j-k|<1$. Thus $j=k$ and the two paths come from edges which share a vertex.

If $a\in(0,1)$ then $w=x\in\{i,j\}$. Since $d_\infty((w,x,a),(y,z,b))\geq \max\{|w-y|,|x-z|\}$ and at least one of $y,z$ is equal to either $k$ or $l$, one of $i,j$ must be equal to one of $k,l$. Thus the two paths come from edges which share a vertex.

If $a=0$ then either $x=i$ or $w=x=j$. Also $b<1$ so either $z=k$ or $y=z=l$. If $x=i$ and $z=k$ then the argument from the $a=1$ case holds. Next, suppose $w=x=j$. Since $z\in\{k,l\}$ and $d_\infty((w,x,a),(y,z,b))\geq |x-z|$, we have $j=k$ or $j=l$. If $x=i$ and $y=z=l$, then $i=l$ following the same reasoning.
\end{proof}

Next, we embed $[0,N-1]^2\times [0,1]$ into $\HH_\R^3$. We work in the upper-half space model of $\HH_\R^3=\setcon{(x,y;z)}{z>0}$. 

Consider the map $\phi:\R^2\times [0,1]\to \HH_\R^3$ defined by
\[
 (x,y,a) \mapsto (x,y;h_0e^{-a}).
\]
\begin{lemma} Let $f:K_N\to[0,N-1]^2\times[0,1]$ be the $1$-thick topological embedding from Lemma \ref{fineGtoR2}. The map $g=\phi\circ f$ is a $1$-thick embedding of $K_N$ into $\HH_\R^3$ with diameter $\leq 2\ln N+9$ and volume $\leq 2039N^2$.
\end{lemma}
\begin{proof}
We first prove that $g$ is $1$-thick. Since $f$ is $1$-thick with respect to the $L^\infty$ metric, it suffices to prove that $d_{\HH_\R^3}(\phi(a_1,b_1,c_1),\phi(a_2,b_2;c_2))\geq 1$ whenever $(a_1,b_1,c_1),(a_2,b_2,c_2)\in [0,n-1]^2\times[0,1]$ are at $L^\infty$ distance $\geq 1$.

Suppose $\max\{|a_2-a_1|,|b_2-b_1|,|c_2-c_1|\}\geq 1$. If $\max\{|a_2-a_1|,|b_2-b_1|\}\geq 1$, then
\[
 d_{\HH_\R^3}(\phi(a_1,b_1,c_1),\phi(a_2,b_2,c_2)) \geq \cosh^{-1}\left(1 + \frac{1}{2h_0^2}\right) = 1.
\]
If $|c_2-c_1|\geq 1$, then
\begin{eqnarray*}
d_{\HH_\R^3}(\phi(a_1,b_1,c_1),\phi(a_2,b_2,c_2)) &\geq & \cosh^{-1}\left(1+ \frac{h_0^2(1-e^{-1})^2}{2h_0^2e^{-1}}\right) \\ & = & \cosh^{-1}(\cosh(1))=1.
\end{eqnarray*}
Next we bound the diameter and the volume. For every point $(x,y;z)$ in the image of $g$, we have $|x|,|y|\leq N-1$ and $h_{1}=h_0e^{-1} \leq z \leq h_0$. Thus
\begin{eqnarray*}
 d_{\HH_\R^3}((0,0;h_0),(x,y;z)) & \leq & \cosh^{-1}\left(1+ \frac{2(N-1)^2+h_0^2(1-e^{-1})^2}{2h_0^2e^{-2}}\right) 
 \\
 & \leq &
 \cosh^{-1}\left(1+ \frac{2e^2N^2+e^2h_0^2}{2h_0^2}\right)
  \\
 & = &
 \cosh^{-1}\left(1+ \frac{e^2}{2} + 2e^2(\cosh(1)-1)N^2\right)
   \\
 & \leq &
 \cosh^{-1}\left(2e^2\cosh(1)N^2\right)
   \\
 & \leq &
 \ln\left(4e^2\cosh(1)N^2\right)
   \\
 & = &
 2\ln(N) + \ln(4e^2\cosh(1)) \leq 2\ln(N)+9.
\end{eqnarray*}
Next, we bound the volume. For each point $(x,y;z)$ in the image of $g$ there is a point $(a,b;h_0)$ with $a,b\in \{0,\ldots,N-1\}$ such that $|x-a|\leq \frac12$, $|y-b|\leq\frac12$ and $z\in[h_0e^{-1},h_0]$. We have
\begin{eqnarray*}
 d_{\HH_\R^3}((a,b;h_0),(x,y;z)) & \leq & \cosh^{-1}\left(1+ \frac{\frac12^2+\frac12^2+h_0^2(1-e^{-1})^2}{2h_0^2e^{-2}}\right)
 \\
 & \leq & 
 \cosh^{-1}\left(1 + \frac{1}{4h_0^2e^{-2}}+\frac{1}{2e^{-2}}\right) \\
 & = & \cosh^{-1}\left(1 + \frac{e^2\cosh(1)}{2}\right) =: \lambda.
 \end{eqnarray*}
Thus, the volume of the $1$-neighbourhood of the image of $g$ is at most $CN^2$ where $C$ is the volume of the ball of radius $\lambda+1$ in $\HH_\R^3$. We have
\[
 C=\pi(\sinh(2(\lambda+1))-2(\lambda+1)) \leq 2039
\]
as required.
\end{proof}

Using the same strategy, we can also prove the following.

\begin{theorem*}\label{thm:thickRH2UB}
There is a constant $C$ such that for every $N\in\N$, there is a $1$-thick topological embedding $g:K_N\to\R\times\HH^2_{\R}$ with $\diam(g)\leq CN$ and $\vol(g)\leq CN^2$.
\end{theorem*}
\begin{proof}
Repeat the proof of Theorem \ref{thm:thickH3ub} but replace the map $\phi$ by
\[
\phi:\R^2\times[0,1]\to \R\times \HH^2_{\R} \quad \textup{given by} \quad \phi(x,y,z)=(x;y,h_0e^{-z}).\qedhere
\]
\end{proof}

\section{Coarse wiring}
In this section, we present some elementary properties of coarse wirings and construct coarse wirings of finite graphs into a Cayley graph of the lamplighter group $\Z_2\wr\Z$.

Recall that a map $r:X\to Y$ between metric spaces is \textbf{$\kappa$-regular} if $d_Y(r(x),r(y))\leq \kappa(1+d_X(x,y))$ for all $x,y\in X$ and the preimage of every ball of radius $1$ in $Y$ is contained in a union of at most $\kappa$ balls of radius $1$ in $X$. We will first prove Proposition \ref{prop:regular}, we recall the statement for convenience.

\begin{proposition*}
  Let $X$ and $Y$ be graphs with maximal degree $\Delta$ and let $r: VX\to VY$ be a
  $\kappa$-regular map. Then for all sufficiently large $k$ we have
  \begin{align*}
    \wir^k_{X\to Y}(n) \leq \left(\kappa+\frac12\right)\Delta n.
  \end{align*}
\end{proposition*}

\begin{proof}
  Let $\Gamma
  \subset X$ be a subgraph with $\abs{V\Gamma}\leq n$.  For $xx' \in E\Gamma$
  let $P_{xx'}$ be any minimal length path from $r(x)$ to $r(x')$ and let
  $\Gamma' = \bigcup_{E\Gamma} P_{xx'}$. We construct a wiring $f:\Gamma\to \Gamma'$ as follows. For each vertex $v\in V\Gamma$ we define $f(v)=r(v)$. We then map each edge $xx'$ continuously to the path $P_{xx'}$.
  
  Since each path $P_{xx'}$ contains at most $2\kappa+1$ vertices and $|E\Gamma|\leq \frac12\Delta n$, we have $\abs{V\Gamma'} \leq n\Delta(\kappa+\frac12)$.

  If $P_{xx'}$ contains an edge $e$ then the distance from $r(x)$ to the initial
  vertex of $e$ is at most $2\kappa$, so there are at most $1+\Delta^{2\kappa+1}$
  possibilities for $r(x)$; $r$ is at most $\kappa(1+\Delta)$-to-one so there are at most
  $k:=\kappa(1+\Delta)(1+\Delta^{2\kappa+1})$ possibilities for $x$. Therefore there are at most
  $k$ edges $xx' \in E\Gamma$ such that $f(xx')=P_{xx'}$
  contains a given edge $e$ of $E\Gamma'$. It follows that
  $\wir^k(\Gamma\to Y) \leq (\kappa+\frac12)\Delta n$.
\end{proof}

To deduce Corollary \ref{cor:wirreg} from the above proposition, we prove a bound on compositions of coarse wirings (Proposition \ref{prop:composition}).

\begin{proposition*}
  \label{lem:composition}
  Suppose $\wir^k_{X\to Y}(N)<\infty$. Then
  \begin{align*}
    \wir^{kl}_{X\to Z}(N) \leq \wir^l_{Y\to Z}\left(\wir^k_{X\to Y}(N)\right).
  \end{align*}
\end{proposition*}

\begin{proof}
  If $\wir^l_{Y\to Z}\left(\wir^k_{X\to Y}(N)\right)=+\infty$ then there is nothing to prove, so assume it is finite. Let $\Gamma \subset X$ with $\abs{V\Gamma}\leq N$. Then there exists a
  coarse $k$-wiring $\psi$ of $\Gamma$ into $Y$ with
  $\vol(W)\leq\wir^k_{X\to Y}(N)$ and a coarse $l$-wiring $\psi'$ of $\im(W)$ into $Z$
  with
  $\vol(W') \leq \wir^l_{Y\to Z}\left(\wir^k_{X\to Y}(N)\right)$.

We now construct a coarse $kl$-wiring $\psi''$ of $\Gamma$ into $Z$. For each $v\in V\Gamma$, define $\psi''(v)=\psi'(\psi(v))$. For each $e\in E\Gamma$, let $e_1,\ldots,e_m$ be the edge path $P_e$. We define $P''_e$ to be the concatenation of paths $P'_{e_1}P'_{e_2}\ldots P'_{e_m}$. We extend $\psi''$ continuously so that the image of $e$ is $P''_e$.
It is clear that $\psi''|_{V\Gamma}$ is $\leq kl$-to-$1$ and $\im(\psi'')\subseteq \im(\psi')$, so $\vol(\psi'')\leq \vol(\psi')$. Since each edge in $\im(\psi'')$ is contained in at most $l$ of the paths $P'_{e'}$ and each $P'_{e'}$ is used in at most $k$ of the paths $P_e$, we have that each edge in $\im(\psi'')$ is contained in the image of at most $kl$ of the edges in $E\Gamma$, as required.
\end{proof}

\begin{proof}[Proof of Corollary \ref{cor:wirreg}]
  This follows immediately from Propositions~\ref{prop:regular} and~\ref{prop:composition}.
\end{proof}

Finally in this section we prove Theorem \ref{thm:wirDL22} by constructing coarse wirings into a Cayley graph of the lamplighter group. This construction is crucial for Theorem \ref{thm:thickemb}. We identify $\Z_2\wr\Z$ with the semidirect product $\bigoplus_{\Z} \Z_2 \rtimes \Z$ and define $Y$ to be the Cayley graph of $\Z_2\wr\Z$ using the generating set $\{(\delta_0,0),(0,1),(0,-1)\}$ where $\delta_0(i)=1$ if $i=0$ and $0$ otherwise. Let us recall the statement of Theorem \ref{thm:wirDL22}.

\begin{theorem*} Let $\Gamma$ be an $n$-vertex graph with maximal degree $d$. There is a coarse $2d$-wiring of $\Gamma$ into $Y$ with diameter at most $6\lceil\log_2(n)\rceil$ and volume at most $dn\left(3\lceil\log_2(n)\rceil+\frac12\right)$.
\end{theorem*}

\begin{proof}
Set $k=\lceil\log_2(n)\rceil$. For each $0\leq i \leq n-1$ and $0\leq j \leq k-1$ fix $i_j\in\{0,1\}$ such that $\sum_{j=0}^{k-1} 2^ji_j=i$.

Enumerate the vertices of $\Gamma$ as $v_0,\ldots,v_{n-1}$. All the points in the image of the wiring will have their lamplighter position and lamp functions supported on the set $\{0,\ldots,2k-1\}$, so we represent elements of $\Z_2\wr\Z$ by a binary string of length exactly $2k$ (for the element of $\bigoplus_{\Z}\Z_2$) with one entry marked by a hat ($\hat{\ }$) to indicate the position of the lamplighter (for the element of $\Z$). Note that this set has diameter at most $6k=6\lceil\log_2(n)\rceil$.

Now we map each $v_i$ to $\hat{i_{0}}i_1\ldots i_{k-1}i_0i_1\ldots i_{k-1}$ and for each edge $v_iv_j$ we assign the path $P_{ij}$ which travels from left to right correcting the binary string as it goes, then returns to the leftmost position:
\begin{eqnarray}
\hat{i_{0}}i_1\ldots i_{k-1}i_0i_1\ldots i_{k-1} & \to & \widehat{j_{0}}i_1\ldots i_{k-1}i_0i_1\ldots i_{k-1} \label{wr1}\\
& \to & j_{0}\widehat{i_1}\ldots i_{k-1}i_0i_1\ldots i_{k-1}  \label{wr2}\\
\ldots & \to & j_{0}j_1\ldots \widehat{j_{k-1}}i_0i_1\ldots i_{k-1} \label{wr3}\\
\ldots & \to & j_{0}j_1\ldots j_{k-1}j_0j_1\ldots \widehat{j_{k-1}} \label{wr4}\\
& \to & j_{0}j_1\ldots j_{k-1}j_0j_1\ldots \widehat{j_{k-2}}j_{k-1} \label{wr5}\\
\ldots & \to & \widehat{j_{0}}j_1\ldots j_{k-1}j_0j_1\ldots j_{k-1}.\label{wr6}
\end{eqnarray}
Now suppose an edge $e$ lies on one of the paths $P_{ij}$. Choose one of the end vertices and denote the binary string associated to this vertex by $a_0\ldots a_{2k-1}$. We claim that at least one of the following holds:
\[
 i = \sum_{l=0}^{k-1} 2^la_{k+l} \ (\dagger) \quad\quad j = \sum_{l=0}^{k-1} 2^la_l \ (\ddagger)
\]
In particular, as the graph $\Gamma$ has maximal degree at most $d$, this means that there are at most $2d$ paths containing the edge $e$.

If $e$ appears on $P_{ij}$ during stages (\ref{wr1}), (\ref{wr2}) or (\ref{wr3}), then $a_{k+l}=i_l$ for $0\leq l \leq k-1$. Thus $(\dagger)$ holds. Otherwise, $e$ appears on $P_{ij}$ during stages (\ref{wr4}), (\ref{wr5}) or (\ref{wr6}), then $a_{l}=j_l$ for $0\leq l \leq k-1$. Thus $(\ddagger)$ holds.

For the volume estimate, each path $P_{ij}$ meets at most $6k+1$ vertices and there are $|E\Gamma|\leq\frac12nd$ paths.
\end{proof}

\section{From fine wirings to coarse wirings and back}\label{sec:thicktocoarse}

In this section we prove Proposition~\ref{prop:coarse_wiring} and
Theorem~\ref{thm:thickemb}, which describe circumstances in which one
can translate between thick embeddings of a graph into a metric space and
coarse wirings of that graph into a graph quasi-isometric to the metric space.

\subsection{Fine to coarse}
In this subsection we will prove Proposition \ref{prop:coarse_wiring}. 

\begin{definition}\label{defn:controlledgrowth}
	Let $\mu$ be a measure on a metric space $X$. We say $(X,\mu)$ has \textbf{controlled growth} if for every $r>0$
	\[
	 c_r:=\inf_{x\in X} \mu(B_r(x))>0 \quad \textup{and} \quad C_r:=\sup_{x\in X} \mu(B_r(x))<+\infty.
	\]
\end{definition}

Let us recall the statement.

\begin{proposition*} Let $M$ be a Riemannian manifold with controlled growth and let $Y$ be a graph
    quasi-isometric to $M$. For any $d \in \N$ and $T>0$, there exists a constant $k$
    depending only on $d$, $M$, $T$ and $Y$ such that if $\Gamma$ is a finite graph
    with maximal degree $d$ and there is a $T$-thick embedding $\phi:\Gamma\to M$ with diameter $D$ and volume $V$ then there is
    a coarse $k$-wiring of $\Gamma$ into $Y$ with diameter at most $kD$ and volume at most $kV$.
\end{proposition*}
\begin{proof}
    Let $f\colon M\to Y$ be a (possibly discontinuous) quasi-isometry. Let
    $\lambda\geq 1$ be such that
    \begin{enumerate}
        \item $\frac{1}{\lambda}d_Y(f(x_1), f(x_2))-\lambda\leq d_M(x_1,x_2) \leq \lambda d_Y(f(x_1), f(x_2)) + \lambda$ for $x_1$ and $x_2$ in $M$, and
        \item for any $y \in Y$, there exists $x \in M$ with $d_Y(y, f(x)) \leq \lambda$.
    \end{enumerate}
    We show that $f\phi$ can be perturbed to obtain a coarse wiring $\psi$.

    For $v \in V\Gamma$, let $\psi(v)$ be any vertex of $Y$ within distance $\frac12$
    of $f\phi(v)$. If $w$ is another vertex of $\Gamma$ with $\psi(w) =
    \psi(v)$ then $d_M(\phi(v), \phi(w)) \leq 3\lambda$. But, for any distinct pair of vertices $v,w$, $d_M(\phi(v),
    \phi(w)) \geq T$, so it follows that at most $C_{3\lambda + T/2}/c_{T/2}$
    vertices of $\Gamma$ map under $\psi$ to $\psi(v)$.

    We now describe a collection of paths $P_{vv'}$ in $Y$ as
     $v$ and $v'$
    range over pairs of adjacent vertices in $\Gamma$. The restriction of
    $\phi$ to the edge $vv'$ is a continuous path in $M$; choose a sequence
    $\phi(v) = w_0', \dotsc, w_n' = \phi(v')$ of points on this path with $n$
    minimal such that $d(w_i', w_{i+1}') \leq 2T$ for each $i$. Denote this minimal $n$ by $n_{vv'}$. Choose
    $w_0 = \psi(v)$, $w_n = \psi(v')$ and for each $1\leq i\leq n-1$ let $w_i$ is a vertex of $Y$ within
    distance $\frac12$ of $f(w_i')$. For each $i$ we have 
    \begin{eqnarray*}
    d_Y(w_i, w_{i+1}) & \leq & 1+d_Y(f(w'_i),f(w'_{i+1})) \\ & \leq & 1 + \lambda d_M(w'_i,w'_{i+1})+\lambda^2 \\ & \leq & 1 + 2\lambda T +\lambda^2:=L,
	\end{eqnarray*}
    so can be joined by an edge path comprising at most $L$
    edges. We define the path $P_{vv'}$ to be the concatenation of these $n_{vv'}$ paths of length at most $L$.

    We extend $\psi$ to a continuous map which sends each edge $vv'$ to the path $P_{vv'}$. We claim that $\psi$ is a
    coarse wiring with the appropriate bounds on diameter and volume. 
    
    Firstly, we bound the diameter. Note that every point in $\im(\psi)$ is within distance $(L+1)/2$ of some $f(w')$ with $w'\in\im(\phi)$. Let $x,y\in \im(\psi)$ and let $v,w \in \Gamma$ satisfy $d_Y(x,f\phi(v)),d_Y(y,f\phi(w))\leq (L+1)/2$. We have
    \begin{eqnarray*}
     d_Y(x,y) & \leq & d_Y(x,f\phi(v)) + \lambda\left(d_M(\phi(v),\phi(w))+\lambda\right) + d_Y(f\phi(w),y) \\
     & \leq & L+1 +\lambda.\diam(\phi)+\lambda^2 \\
     & \leq & C(T,\lambda).\diam(\phi).
    \end{eqnarray*}
The final inequality fails if $\Gamma$ is a single vertex, but the proposition obviously holds in this situation. Otherwise $\diam(\phi)\geq T$ and the inequality holds for a suitable $C$.

    Next we bound the volume of the wiring. The bound follows from the two inequalities 
    \[
    \vol(\phi)\geq \frac{c_{T/2}}{2d+1}\left(|V\Gamma|+ \sum_{vv'\in E\Gamma} n_{vv'}\right) \quad \textup{and} \quad \vol(\psi) \leq |V\Gamma| + L\sum_{vv'\in E\Gamma} n_{vv'}.
	\]
For the second bound, each vertex in $V\Gamma$ contributes at most $1$ vertex to $\vol(\psi)$ and each path $P_{vv'}$ contributes at most $Ln_{vv'}$ vertices to $\vol(\psi)$. For the first bound, notice that the (open) balls of radius $T/2$ around the image of each vertex are necessarily disjoint. Similarly, the balls of radius $T/2$ centred at any two points in one of the sequences $\phi(v) = w_0', \dotsc, w_n' = \phi(v')$ defined above are necessarily disjoint: if this were not the case for $w'_j$ and $w'_{j'}$, we must have $|j-j'|\geq 2$ since $d(w'_i,w'_{i+1})\geq T$ for all $i$, but then we can remove $w'_{j+1},\ldots,w'_{j'-1}$ from the above sequence, contradicting the minimality assumption. Moreover, if two balls of radius $T/2$ centred at points on sequences corresponding to different edges have non-trivial intersection, then these edges must have a common vertex since $\phi$ is a $T$-thick embedding. Thus, the $T$-neighbourhood of the image of $\phi$ contains a family of $\left(|V\Gamma|+ \sum_{vv'\in E\Gamma} n_{vv'}\right)$ balls of radius $T/2$, such that no point is contained in more than $2d+1$ of these balls ($d$ for each end vertex, and an extra $1$ if the point is within distance $T/2$ of the image of a vertex). As a result
\[
 \vol(\phi)\geq \frac{c_{T/2}}{2d+1}\left(|V\Gamma|+ \sum_{vv'\in E\Gamma} n_{vv'}\right).
\]
	    
    It remains to prove that we have defined a coarse $k$-wiring. It is sufficient to show that there is a constant $k$
    depending only on $\lambda$ and the growth rates $c$ and $C$ of volumes in
    $M$ such that any edge of $Y$ is contained in $P_{vv'}$ for at most
    $k$ edges $vv' \in E\Gamma$.

    Let $uu'$ be an edge of $Y$ contained in at least one path in the
    collection $P$. Let $A$ be the subset of $E\Gamma$ comprising edges $e$
    such that $P_{e}$ contains $uu'$. As noted during the proof of the diameter bound every point in $P_{e}$ is contained in the $(L+1)/2$-neighbourhood of $f(\phi(e))$ so there is a point 
    $x_e \in \phi(e)$ such that $d_Y(u, f(x_e)) \leq (L+1)/2$, and so for
    any other edge $e' \in A$, 
    \begin{align*}
        d_M(x_e, x_{e'}) \leq \lambda \left(d_Y(f(x_e),u)+d_Y(u,f(x_{e'}))\right)+\lambda  \leq \lambda(L+2).
    \end{align*}

    For any edge $e' \in A$, $x_{e'}$ is within distance $T$ of at most
    $2d$ of the points $x_{e''}$ for $e'' \in A$. It follows that the size of $A$
    is at most $2d c_{T/2}^{-1}C_{\lambda(L+2)+T/2}$.
\end{proof}

\subsection{Coarse to fine}

The return direction is more sensitive and we are not able to obtain $1$-thick
embeddings in all cases. When the target space is Euclidean this is easily
resolved by rescaling, but in other spaces changing thickness potentially has a
more drastic effect on the volume. Let us recall Theorem \ref{thm:thickemb}.

\begin{theorem*}
	Let $M$ be a compact Riemannian manifold of dimension $n\geq 3$, let $Y$ be a graph quasi-isometric to the universal cover $\widetilde{M}$ of $M$ and let $k,d\in\N$. There exist constants
	$C$ and $\varepsilon>0$ such that the following holds:
	
	If there is a coarse $k$-wiring of a finite graph $\Gamma$ with maximal degree $d$ into $Y$ with diameter $D$ and volume $V$ then there is a $\varepsilon$-thick embedding
	of $\Gamma$ into $\widetilde{M}$ with diameter at most $CD$ and volume at most $CV$.
\end{theorem*}

The proof of this result is completed in several steps. As we are aiming to construct a topological embedding, the first step is to replace the coarse $k$-wiring $\Gamma\to Y$ with an injective continuous function $\Gamma\to Y'$ where $Y'$ is a ``thickening'' of $Y$. Exploiting the symmetries in the universal cover, we choose $Y$ (and its thickening) to be cocompact with respect to the action of $\pi_1(M)$, this reduces the problem of embedding the thickening of $Y$ to defining finitely many paths in $M$. We then use the fact that $M$ is compact to obtain a positive lower bound on the thickness of the topological embedding. 

Using Proposition \ref{prop:regular} and the fact that quasi-isometries of bounded degree graphs are regular, it suffices to prove
Theorem~\ref{thm:thickemb} for a specific bounded degree graph
quasi-isometric to $\widetilde M$.

We require a standard \u{S}varc-Milnor lemma.
\begin{lemma}\label{lem:MilnorSvarc}
    Let $x \in M$. Then, for sufficiently large
    $L$, the graph $\calG_x^L$ with vertex set equal to the preimage of $x$ in
    $\widetilde M$, with vertices connected by an edge if and only
    if they are separated by a distance of at most $L$ in $\widetilde M$, is quasi-isometric to $\widetilde M$.
\end{lemma}

Now we assume that $Y=\calG_x^L$ for a suitably chosen $L$. The next step is to ``thicken'' $Y$ to a graph $Y'$ to obtain injective wirings.

\begin{definition} A wiring $f:\Gamma\to Y$ of a finite graph $\Gamma$ into a graph $Y$ is called an \textbf{injective wiring} if $f$ is injective.
\end{definition}

\begin{definition} Given a graph $Y$ and $T\in\N$ we define the \textbf{$T$-thickening} of $Y$ to be the graph $K_T(Y)$ with vertex set $VY\times\set{1,\ldots,T}$ and edges $\{(v,i),(w,j)\}$ whenever either $v=w$ and $1\leq i < j \leq T$, or $\{v,w\}\in EY$ and $1\leq i \leq j \leq T$.
\end{definition}

\begin{lemma}\label{lem:injwiring} For all $d,k\in\N$ there exists some $T$ with the following property. If there is a coarse $k$-wiring $\psi:\Gamma\to Y$ then there is an injective wiring $\psi':\Gamma\to K_T(Y)$, such that $\diam(\psi')\leq \diam(\psi)+2$ and $\vol(\psi') \leq T \vol(\psi)$.
\end{lemma}

\begin{proof} Set $T=k(d+1)$. For each vertex $v\in Y$ enumerate $\psi^{-1}(v)=\set{v_1,\ldots, v_{m}}$ for some $m\leq k$. Define $\psi'(v_l)=(\psi(v),l)$. We now define $\psi'$ on the edges of $\Gamma$. Whenever $vw$ is an edge in $\Gamma$ with $\psi(v)=\psi(w)$ then there exist $l\neq l'$ such that $v=v_l$ and $w=v_{l'}$ and we map the corresponding interval in $\Gamma$ with endpoints labelled $v$ and $w$ isometrically onto the interval with endpoints labelled $(\psi(v),l)$ and $(\psi(v),l')$ in $K_T(Y)$.

Define $E'$ to be the set of edges in $\Gamma$ whose end vertices are distinct after applying $\psi$. Enumerate the edges of $E'$ as $e_1=v_1w_1,\ldots,e_n=v_nw_n$ and, for each $j$, enumerate (in order, including non-consecutive repetitions) the vertices contained in the image of the interval corresponding to $v_jw_j$ under $\psi$ as $\psi(v_j)=v_j^0,\ldots,v_j^{n_j}=\psi(w_j)$. Each $v_j^iv_j^{i+1}$ is an edge in $Y$.

Completely order the $v^i_j$ so that $v^i_j<v^{i'}_{j'}$ whenever $j<j'$ or $j=j'$ and $i<i'$. Denote by $n(v^i_j)$ the number of $v^{i'}_{j'}<v^i_j$ such that $v^{i'}_{j'}$ and $v^i_j$ are the same vertex of $Y$.  We extend $\psi'$ by mapping the edge $e_j$ continuously and injectively onto the path $(v_j^0,n(v_j^0)+k+1),(v_j^1,n(v_j^1)+k+1),\ldots, (v_j^{n_j},n(v_j^{n_j})+k+1))$. It is immediate from the construction that $\psi'$ is an injective wiring. It remains to find a uniform bound on $n(v_j^{i})$.

Since $\psi$ is a coarse $k$-wiring, each vertex in $Y$ lies in the interior of at most $kd$ of the $\psi(e)$, so $n(v_j^{i})\leq kd-1$ for all $i,j$. Thus $\psi'$ as above is well-defined since $T= k(d+1)$.

Note that if $(x,j),(y,j')$ are contained in $\im(\psi')$ then $(x,1),(y,1)\in\im(\psi')$ and there is a path of length at most $\diam(\psi)$ connecting $(x,1)$ to $(y,1)$ in $K_T(Y)$. Hence $\diam(\psi')\leq\diam(\psi)+2$. If $(x,j)\in \im(\psi')$ for some $j$ then $x\in\im(\psi)$. Therefore $\vol(\psi')\leq T\vol(\psi)$.
\end{proof}

The next step is to find a thick embedding of $K_T(Y)$ into $\widetilde{M}$.

\begin{lemma}\label{lemma:equivariantembedding}
    Suppose that $M$ is a compact manifold of dimension $n\geq 3$ with fundamental group $G$ and let $\widetilde M$ be the universal cover of $M$. Let $x \in M$ and denote by $Gx$
    the orbit of $x$ in $\widetilde M$ under $G$. Then for any $L,T$ there is an embedding of
    $Y'=K_T(\mathcal{G}^L_X(Gx))$ into $\widetilde M$ that is equivariant with respect to the
    action of $G$ on $\mathcal M$ by deck transformations. 

    This embedding is $\varepsilon$-thick for some $\varepsilon>0$, and there is a uniform upper bound on the length of the paths obtained as the images of edges of $Y'$ under the embedding.
\end{lemma}

\begin{proof}
	Let $B$ be a ball in $M$ centred at $x$ which is homeomorphic to $\R^n$. Fix a topological embedding $f$ of the complete graph on $T$ vertices into $B$. Enumerate the vertices $\{w_1,\ldots,w_T\}$. For each pair $w_a,w_b\in VK_T$, and each homotopy class $[\ell]$ in $\pi_1(M,x)$ which has a representative of length at most $L$, choose an arc $\gamma_{a,b,[\ell]}$ connecting $f(w_a)$ to $f(w_b)$ such that the loop obtained from concatenating $f(w_aw_b)$ and $\gamma_{a,b,[\ell]}$ is in $[\ell]$ and such that $\gamma_{a,b,[\ell]}$ intersects the union of $f(K_T)$ and all arcs previously added only at the points $f(y)$ and $f(z)$. This is always possible using a general position argument.
	
	Lifting this embedding to $\overline M$, we obtain a $G$-equivariant embedding of $K_T(\mathcal{G}^L_X(Gx))$ into $\widetilde M$. Specifically, the interval with end points $(gx,a)$ and $(gx,b)$ is mapped to $gf(w_aw_b)$ and if $(gx,a)(g'x,b)$ is an edge in $Y'$ then by definition the homotopy class corresponding to $g=[\ell]$ has a representative of length at most $L$. Thus, the image of this edge in $\widetilde{M}$ is the lift of $\gamma_{a,b,[\ell]}$ starting at $(gx,a)$ and ending at $(g'x,b)$. 
	As the natural covering map $\widetilde{M}\to M$ is $1$-Lipschitz, this topological embedding is $\varepsilon$-thick, where $\varepsilon=\min\set{d_M(X,Y)}$ as $X,Y$ range over the following:
	\begin{itemize}
	 \item $X=\{f(v)\}$, $Y=\{f(w)\}$ for distinct vertices $v,w\in VK_T$; or
	 \item $X=\{f(v)\}$ and $Y$ is either $f(yz)$ or $\gamma_{y,z,[\ell]}$ with $v,y,z$ all distinct; or
	 \item $X$ is either $f(vw)$ or $\gamma_{v,w,[\ell]}$ and $Y$ is either $f(yz)$ or $\gamma_{y,z,[\ell']}$ with $v,w,y,z$ all distinct.
	\end{itemize}
	Similarly, since there are only finitely many $G$-orbits of images of edges, there is a uniform upper bound on the lengths of images of edges.
\end{proof}

We are now ready to prove Theorem \ref{thm:thickemb}.

\begin{proof}[Proof of Theorem \ref{thm:thickemb}]
Let $M$ be a compact manifold of dimension $n\geq 3$, let $\widetilde M$ be the universal cover of $M$ and let $Y$ be any graph quasi-isometric to $\widetilde M$. Fix $d,k\in\N$ and assume that there is a coarse $k$-wiring of $\Gamma$ into $Y$ with diameter $D$ and volume $V$. We may assume $D\geq 1$ as the $D=0$ case is obvious.

By Lemma \ref{lem:MilnorSvarc} there is some $L$ such that $\calG_x^L$ is quasi-isometric to $\widetilde M$, so by Corollary \ref{cor:wirreg}(\ref{eq:cwreg1}), there exists some $l=l(k,d)$ so that there is a coarse $l$-wiring of $\Gamma$ into $\calG_x^L$ with diameter $\leq lD+l\leq 2lD$ and volume $\leq lV+l\leq 2lV$.

Now we apply Lemma \ref{lem:injwiring}: for some $T=T(l,d)$ there is an injective wiring $\psi$ of $\Gamma$ into $K_{T}(\calG_x^L)$ with diameter $\leq 2lD+2\leq 4lD$ and volume $\leq 2TlV$. Composing this injective wiring with the $\varepsilon$-thick topological embedding $\phi$ of $K_{T}(\calG_x^L)$ into $\widetilde M$ gives an $\varepsilon$-thick embedding $f:\Gamma\to \widetilde{M}$.
The diameter of the image of $f$ is bounded from above by a constant multiple of $\diam(\psi)$. For the volume, note that the sum of the lengths of all paths in the wiring is at most a constant times $\vol(\psi)$, and the volume of the thick embedding is at most this sum of lengths multiplied by the maximal volume of a ball of radius $\varepsilon$ in $M$. Hence the volume of this thick embedding is at most a constant multiple of $V$.
\end{proof}

\section{Lower bounds on coarse wiring}
The goal of this section is to prove Theorem \ref{thm:wirsep}. We begin by recalling the definition of the separation profile and its key properties.

\subsection{Background on the separation profile}

Recall that $f\lesssim g$ if there is a constant $C$ such that
\[
f(x)\leq Cg(Cx)+C \quad \textup{for all}\ x\in X.
\]
We write $f\simeq g$ if $f\lesssim g$ and $g\lesssim f$.

\begin{definition}\cite{BenSchTim} Let $\Gamma$ be a finite graph. We denote by $\cut(\Gamma)$ the minimal cardinality of a set $S\subset V\Gamma$ such that no component of $\Gamma-S$ contains more than $\frac12|V\Gamma|$ vertices. A set $S$ satisfying this property is called a \textbf{cut set} of $\Gamma$.

Let $X$ be a (possibly infinite) graph. We define the \textbf{separation profile} of $X$ to be the function $\sep_X:[0,\infty)\to[0,\infty)$ given by
\[
\sep_X(n)=\max\{\cut(\Gamma)\mid \Gamma\leq X,\ |V\Gamma|\leq n\}.
\]
For convenience, we will define $\sep_X(r)=0$ whenever $r<1$.
\end{definition}

\begin{definition}\label{defn:Cheeger} The Cheeger constant of a finite graph $\Gamma$ is
\[
h(\Gamma)=\min\{\frac{|\partial A|}{|A|} \mid A\subseteq V\Gamma,\ |A|\leq\frac12|V\Gamma|\}
\]
where $\partial A = \{v\in V\Gamma\mid d_\Gamma(v,A)=1\}$.
\end{definition}

\subsection{Lower bounds on wiring profiles}

The key part of the proof of Theorem \ref{thm:wirsep} is the following intimidating bound.

\begin{proposition}\label{prop:sepLB} Let $X,Y$ be graphs with maximal degrees $\Delta_X,\Delta_Y$ respectively. If $\wir^k_{X\to Y}(n)<\infty$, then, for all $n\geq 3$,
\[
 \sum_{s\geq 0}\sep_Y(2^{-s}\wir^k_{X\to Y}(n)) \geq \frac{\sep_X(n)}{k\Delta_Y}.
\]
\end{proposition}

Roughly, the idea is that given a subgraph $\Gamma\leq X$ and an ``efficient'' coarse $k$-wiring $\Gamma\to\Gamma'$, $\cut(\Gamma')$ can be bounded from above by $\cut(\Gamma)$ up to a multiplicative error depending only on $k$. However, we do not know that any cut of this size equally divides the images of the vertices of $\Gamma$ in $\Gamma'$ so we may need to repeat the procedure on a subgraph of $\Gamma'$ with at most $|\Gamma'|/2$ vertices and then again on a subgraph of $\Gamma'$ with at most $|\Gamma'|/4$ vertices and so on. This divide-and-conquer strategy is the reason for the summation on the left hand side of the equation above. 

\begin{proof} Let $n\geq 3$ and choose $\Gamma\leq X$ which satisfies $\abs{\Gamma}\leq n$ and $\cut(\Gamma)=\sep_X(n)=l$. Since $n\geq 3$, it is always the case that $l\leq 2\abs{\Gamma}/3$. By \cite[Proposition 2.4]{HumSepExp} there is some $\Gamma''\leq \Gamma$ which satisfies $\abs{\Gamma''}\geq\frac12\abs{\Gamma}$ and $h(\Gamma'')\geq \frac{l}{2\abs{\Gamma}}$. Let $\psi:\Gamma''\to\Gamma'$ be a coarse $k$-wiring where $\Gamma'\leq Y$ satisfies $|\Gamma'|=wir^k_{Y}(\Gamma'')$.

Let us recursively define a collection of subsets $C'_1,C'_2,\ldots, \subseteq V\Gamma'$ as follows.
Define $\Gamma'_0=\Gamma$. Let $C'_s$ be a cut set of $\Gamma'_{s}$ of minimal size. If for every component $A$ of $\Gamma'_{s}-C'_s$, we have $|\setcon{v\in V\Gamma}{\psi(v)\in A}|\leq\frac12|\Gamma$ then define $C'_t=\emptyset$ for all $t>s$ and end the process. Otherwise, set $\Gamma_{s+1}$ to be the unique connected component of $\Gamma'_{s}-C'_s$ satisfying $|\setcon{v\in V\Gamma}{\psi(v)\in A}|>\frac12|\Gamma$. As $|\Gamma|<\infty$ this process will always terminate. Define $C'=\bigcup_{s\geq 0} C'_s$. By construction, for every connected component $A$ of $\Gamma'\setminus C'$, we have $|\setcon{v\in V\Gamma}{\psi(v)\in A}|\leq\frac12|\Gamma$. By definition of cut sets, $|\Gamma'_s|\leq 2^{-s}|\Gamma'|$, so $|C'_s|\leq \sep_Y(2^{-s}|\Gamma'|)=\sep_Y(2^{-s}\wir^k_{X\to Y}(n))$.

Let $C$ be the set of all vertices in $\Gamma$ which are the end vertices of an edge whose image under $\psi$ contains any vertex in $C'$. By construction $C$ is a cutset for $\Gamma$. Since $\psi$ is a $k$-coarse wiring, each edge in $\Gamma'$ lies in the image of at most $k$ edges in $\Gamma$, so each vertex in $\Gamma'$ lies in the image of at most $k\Delta_Y$ edges in $\Gamma$, where $\Delta_Y$ is the maximal degree of the graph $Y$. Thus $|C|\leq k\Delta_Y|C'|$.

Combining these observations we see that
\[
 \cut(\Gamma) \leq |C| \leq k\Delta_Y|C'| \leq k\Delta_Y\sum_{s\geq 0} \sep_Y(2^{-s}\wir^k_{X\to Y}(n))
\]
As this holds for all $\Gamma\leq X$ with $\abs{\Gamma}\leq n$, we deduce that
\[
 \sep_X(n) \leq k\Delta_Y\sum_{s\geq 0} \sep_Y(2^{-s}\wir^k_{X\to Y}(n)).\qedhere
\]
\end{proof}

In practice, the separation profiles of graphs we are interested in here are of the form $n^r\ln(n)^s$ with $r\geq 0$ and $s\in\R$. Restricted to these functions, Proposition \ref{prop:sepLB} says the following.

\begin{corollary}\label{cor:sepLB} Suppose $\sep_X(n)\gtrsim n^r\ln(n)^s$ and $\sep_Y(n)\simeq n^p\ln(n)^q$. Then, for all $k$,
\[
wir^k_{X\to Y}(n)\gtrsim \left\{ 
\begin{array}{lll} 
 n^{r/p}\ln(n)^{(s-q)/p} & \textup{if} & p>0, \\
 \exp(n^{r/(q+1)}\ln(n)^{s/(q+1)}) & \textup{if} & p=0.
 \end{array}\right.
 \]
\end{corollary}
\begin{proof}
If $wir^k_{X\to Y}(n)=+\infty$ there is nothing to prove, so assume this is not the case. Applying our hypotheses to Proposition \ref{prop:sepLB}, we have 
\begin{equation}\label{eq:sepwirbd}
 n^r\ln(n)^s \lesssim \sum_{i\geq 0} \left(2^{-i}\wir^k_{X\to Y}(n)\right)^p\ln\left(2^{-i}\wir^k_{X\to Y}(n)\right)^q.
\end{equation}
If $p>0$, then the sequence $\left(2^{-i}\wir^k_{X\to Y}(n)\right)^p$ decays exponentially as a function of $i$, so
\begin{eqnarray*}
 n^r\ln(n)^s & \lesssim & \ln\left(\wir^k_{X\to Y}(n)\right)^q\sum_{i\geq 0} \left(2^{-i}\wir^k_{X\to Y}(n)\right)^p \\ & \lesssim & \wir^k_{X\to Y}(n)^p\ln\left(\wir^k_{X\to Y}(n)\right)^q.
\end{eqnarray*}
Hence, there is some constant $C>0$ such that 
\begin{equation}\label{eq:wir}
w:=\wir^k_{X\to Y}(n)^p\ln\left(\wir^k_{X\to Y}(n)\right)^q \geq C^{-1}(C^{-1}n)^r\ln(C^{-1}n)^s-C.
\end{equation}
Now suppose $\wir^k_{X\to Y}(n)\leq dn^{r/p}\ln(n)^{(s-q)/p}$. Then
\begin{eqnarray*}
w & \leq & d^pn^r\ln(n)^{s-q}\left(\ln(d)+\frac{r}{p}\ln(n) + \frac{s-q}{p}\ln\ln(n)\right)^q \\
& \leq &
\frac{(2r)^sd^p}{p^s}n^r\ln(n)^{s-q}\ln(n)^q = \frac{(2r)^sd^p}{p^s}n^r\ln(n)^s
\end{eqnarray*}
for sufficiently large $n$. This contradicts $\eqref{eq:wir}$ if $d$ is small enough and $n$ is large enough. Hence,
\[
 \wir^k_{X\to Y}(n) \gtrsim n^{r/p}\ln(n)^{(s-q)/p}.
\]
If $p=0$, then by $\eqref{eq:sepwirbd}$ there is some $C>0$ such that 
\begin{eqnarray*}
C^{-1-r}n^r\ln(C^{-1}n)^s-C & \leq & \sum_{i= 0}^{\ln(\wir^k_{X\to Y}(n))} \ln\left(2^{-i}\wir^k_{X\to Y}(n)\right)^q \\ & \leq & \ln\left(\wir^k_{X\to Y}(n)\right)^{q+1},
\end{eqnarray*}
Hence $\wir^k_{X\to Y}(n)\gtrsim \exp(n^{r/(q+1)}\ln(n)^{s/(q+1)})$.
\end{proof}

\section{Completing Theorems \ref{thm:wirehypxEuclb}, \ref{thm:wirhighrkub} and \ref{thm:wirhighrklb}}\label{sec:results}
In this section we give complete proofs of the main results of the paper.

\begin{proof}[Proof of Theorem \ref{thm:wirehypxEuclb}]
 Let $M=\HH^q_F\times\R^r$ and let $Y$ be a bounded degree graph which is quasi-isometric to $M$. Fix $d\in\N$ and $\delta,\epsilon>0$. Enumerate the set of finite graphs with maximal degree $d$ and Cheeger constant $\geq \delta$ by $\Gamma_1,\Gamma_2,\ldots$, with $|\Gamma_i|\leq|\Gamma_j|$ whenever $i\leq j$. Let $X$ be the disjoint union of all the $\Gamma_i$. If $X$ is finite, there is nothing to prove. By \cite[Proposition 2.4]{HumSepExp}, $\sep_X(|\Gamma_i|) \geq \frac{\delta}{2}|\Gamma_i|$ for all $i$. Set $Q=(q+1)\dim_\R(F)-2$, the conformal dimension of the boundary of $\HH^q_F$. We have $\sep_Y(n)\simeq n^{1-1/(r+1)}\ln(n)^{1/(r+1)}$ if $Q=1$ and $\sep_Y(n)\simeq n^{1-1/(Q+r-1)}$ if $Q\geq 2$, so by Corollary \ref{cor:sepLB}, for each $k$ there exists a constant $C$ such that for all $i$,
\begin{equation}\label{eq:wirlbrk1}
	\wir^k_{X\to Y}(|\Gamma_i|)\geq \left\{ \begin{array}{rcl} C^{-1}|\Gamma_i|^{1+1/r}\ln(1+|\Gamma_i|)^{-1/r}-C & \textup{if} & Q=1, \\ C^{-1}|\Gamma_i|^{1+1/(Q+r-1)}-C & \textup{if} & Q\geq 2. \end{array}
\right.
\end{equation}
We continue with the $Q=1$ case, the argument for $Q\geq 2$ is very similar. Now suppose for a contradiction that for every $n$ there is some $\epsilon$-thick embedding $\Gamma_{i_n}\to M$ with volume at most $\frac1n|\Gamma_i|^{1+1/r}\ln(1+|\Gamma_i|)^{-1/r}$. By Proposition $\ref{prop:coarse_wiring}$, there is a coarse $k$-wiring of $\Gamma_{i_n}$ into $Y$ with volume at most \[\frac{k}{n}|\Gamma_i|^{1+1/r}\ln(1+|\Gamma_i|)^{-1/r}\]
for some $k=k(d,M,\varepsilon,Y)$. This contradicts $\eqref{eq:wirlbrk1}$ for large enough $n$.
\end{proof}

\begin{proof}[Proof of Theorem \ref{thm:wirhighrkub}]
Let $\Gamma$ be a finite graph with maximal degree $d$.
By Theorem \ref{thm:wirDL22} there is a $2d$-coarse wiring of $\Gamma$ into a Cayley graph of $\Z_2\wr\Z$ with volume $\leq 4d|\Gamma|\lceil\log_2(1+|\Gamma|)\rceil$. This Cayley graph is quasi-isometric to $\DL(2,2)$ \cite{Woess-LAMPLIGHTERS_HARMONIC_FUNCTIONS}, and $\DL(2,2)$ quasi-isometrically embeds into any graph $X$ quasi-isometric to a symmetric space whose non-compact factor has rank $\geq 2$ \cite[Proposition 2.8 and Theorem 3.1]{HumeMackTess2}. Thus, for some $l$ we have
\[
 \wir^l_{\Gamma\to X}\leq C'N\ln(1+N).\qedhere
\]
\end{proof}
\begin{proof}[Proof of Theorem \ref{thm:wirhighrklb}] The proof is the same as for Theorem \ref{thm:wirehypxEuclb} replacing \eqref{eq:wirlbrk1} with
\begin{equation*}
	\wir^k_{Y\to X}(|\Gamma_i|)\geq C^{-1}|\Gamma_i|\ln(1+|\Gamma_i|)-C.\qedhere
\end{equation*}
\end{proof}

\subsection{Coarse wirings into two dimensional symmetric spaces}\label{sec:2dim}
In this section we collect results about coarse wirings into $\R^2$ and $\HH^2_\R$. The first is a direct construction of a coarse wiring, the second a thick embedding into $\HH_\R^2\times[0,1]$ which is quasi-isometric to $\HH^2_\R$.

\begin{proposition} Every $N$-vertex graph with maximal degree $d$ admits a coarse $2d$-wiring into the standard $2$-dimensional integer lattice $\Z^2$ with volume at most $N^2$.

Let $X$ be the disjoint union of all finite graphs with maximal degree $3$. For any $k$ there is some $C$ such that $\wir^k_{X\to\Z^2}(n)\geq C^{-1}n^2-C$.
\end{proposition}
\begin{proof}
The second claim follows immediately from Corollary \ref{cor:sepLB} and the fact that $\sep_X(n)\simeq n$ \cite{HumSepExp} and $\sep_{\Z^2}(n)\simeq n^{1/2}$ \cite{BenSchTim}.

Let $\Gamma$ be an $n$-vertex graph with maximal degree $d$. Enumerate the vertices of $\Gamma$ by $v_0,\ldots,v_{n-1}$. We construct a $d$-wiring of $\Gamma$ into $\{0,\ldots,n-1\}^2$ as follows:

Map the vertex $v_k$ to the point $(k,k)$. For each edge $v_iv_j$ (with $i<j$) we define a path $P_{ij}$ which travels horizontally from $(i,i)$ to $(j,i)$, then vertically from $(j,i)$ to $(j,j)$.

To see that this is a $1$-thick embedding, note that if a horizontal edge $(a,b)(a+1,b)$ is in $P_{ij}$ then $b=i$. Similarly, if a vertical edge $(a,b)(a,b+1)$ appears in $P_{ij}$, then $a=j$. Hence, any two paths containing a common edge have a common end vertex. Since, by assumption, there are at most $d$ edges with a given end vertex, we have defined a coarse $2d$-wiring. The volume estimate is obvious.
\end{proof}

\begin{proposition} Let $Y$ be a graph which is quasi-isometric to $\HH_\R^2$ and let $d\in\N$. There are constants $k=k(Y,d)$ and $C=C(Y,d)$ such that any $N$-vertex graph $\Gamma$ with maximal degree $d$ admits a coarse $k$-wiring into $Y$ with volume $\leq CN^2\exp(N)$.

Let $X$ be the disjoint union of all finite graphs with maximal degree $3$. For any $k$ there is some $C$ such that $\wir^k_{X\to Y}(n)\geq C^{-1}\exp(C^{-1}n^{1/2}))-C$.
\end{proposition}
\begin{proof}
The second claim follows immediately from Corollary \ref{cor:sepLB} and the fact that $\sep_X(n)\simeq n$ \cite{HumSepExp} and $\sep_{\HH_\R^2}(n)\simeq \ln(n)$ \cite{BenSchTim}.

We will construct $1$-thick embeddings $K_N\to \HH_\R^2\times [0,1]$ with volume $\leq C'N^2\exp(N)$. Since $Y$ is quasi-isometric to $\HH_\R^2\times [0,1]$, the result will then follow from Proposition \ref{prop:coarse_wiring}.

Firstly, recall the definition of the metric in the upper halfspace model of $\HH_\R^2$:
\[
d_{\HH_\R^2}((x_1,y_1),(x_2,y_2)) = \cosh^{-1}\left(1+ \frac{(x_2-x_1)^2+(y_2-y_1)^2}{2y_1y_2}\right).
\]
We equip $\HH_\R^2\times[0,1]$ with the metric 
\[
d((w,x;a),(y,z;b))=\max\{d_{\HH_\R^2}((w,x),(y,z)),|a-b|\}.
\]
Let us define $h_0:=(2(\cosh(1)-1)^{-1/2}>1$.

\textbf{Claim:} If $d((w,x;a),(y,z;b))<1$ and $x,z\leq h_0$, then $|a-b|<1$,  $|w-y|<1$ and $|\ln(x/h_0)-\ln(z/h_0)|<1$.

\begin{proof}[Proof of Claim]
It is immediate from the definition that $|a-b|<1$. Since $x,z\leq h_0$, 
 \begin{eqnarray*}
 1 & > & d((w,x;a),(y,z;b)) \\ & \geq & d_{\HH_\R^2}((w,x),(y,z)) \\ & \geq & \cosh^{-1}\left(1+ \frac{(w-y)^2}{2h_0^2}\right) \\ & \geq & \cosh^{-1}(1+(\cosh(1)-1)(w-y)^2).
 \end{eqnarray*}
 Hence $(w-y)^2<1$, so $|w-y|<1$. Finally, write $x=h_0e^p$ and $z=h_0e^q$ with $p,q\in\R$. We have
 \begin{eqnarray*}
 1 & > & d((w,x;a),(y,z;b)) \\ & \geq & d_{\HH_\R^2}((w,x),(y,z))\\ & \geq & \cosh^{-1}\left(1+ \frac{h_0^2(e^p-e^q)^2}{2h_0^2e^{p+q}}\right) \\ & = & \cosh^{-1}\left(\frac12(e^{p-q}+e^{q-p})\right)=|p-q|.
 \end{eqnarray*}
 Hence $|\ln(x/h_0)-\ln(z/h_0)|=|p-q|<1$.
\end{proof}

Enumerate the vertices of $K_N$ by $v_0,\ldots,v_{N-1}$. We map $v_i$ to $(i,h_0e^{-i};0)$ where $h_0=(2(\cosh(1)-1)^{-1/2}$. For $i<j$, we connect $(i,h_0e^{-i};0)$ to $(j,h_0e^{-j};0)$ using the path $P_{ij}$ defined as follows:
\begin{eqnarray}
 (i,h_0e^{-i};0) & \to & (j,h_0e^{-i};0) \\ & \to & (j,h_0e^{-i};1) \\ & \to & (j,h_0e^{-j};1) \\ & \to & (j,h_0e^{-j};0)
\end{eqnarray}
 where the first segment lies in the horocircle $y=h_0e^{-i}$ and the others are geodesics.
 
 We first prove that this embedding is $1$-thick. Let $(w,x;a)\in P_{ij}$ and $(y,z;b)\in P_{kl}$ with $d((w,x;a),(y,z;b))<1$. Set $p=\ln(x/h_0)$ and $q=\ln(z/h_0)$. From the claim we have $\max\{|w-y|,|p-q|,|a-b|\}<1$.
 
 If $a=1$, then $b>0$, so $w=j$ and $y=l$.  Since $w,y$ are both integers they must be equal. Thus $j=l$ and the two paths come from edges which share a vertex.
 
If $a\in(0,1)$ then $w=j$ and $p\in\{-i,-j\}$. Note that one of the four equalities $y=k$, $y=l$, $q=-k$, $q=-l$ holds at every point on $P_{kl}$. If it one of the first two, then $\min\{|j-k|,|j-l|\}<1$ and $j\in\{k,l\}$, or if it is one of the last two, then one of $-i,-j$ is equal to one of $-k,-l$. In any case the two paths share an end vertex.

If $a=0$ then either $p=-i$ or $w=j$ and $p=-j$. Also $b<1$ so either $q=-k$ or $y=l$ and $q=-l$. If $p=-i$, then either $q=-k$ in which case $|-i-(-k)|<1$ by the claim, thus $i=k$; or $q=-l$ in which case $i=l$ by the same reasoning. Next, suppose $w=j$ and $p=-j$. Since $q\in\{-k,-l\}$ we have $j=k$ or $j=l$. If $p=-i$, $y=l$ and $q=-l$, then $i=l$ following the same reasoning.

Every point in the image of the embedding is of the form $(x,h_0e^{-y};z)$ where $|x|,|y|\leq n-1$ and $z\in[0,1]$. Set $p=\left(\frac{n-1}{2},h_0;\frac12\right)$. We have
\begin{eqnarray*}
 d\left((x,h_0e^{-y};z),p\right) & \leq & \cosh^{-1}\left(1+ \frac{\left(\frac{n-1}{2}\right)^2 + h_0^2(1-e^{n-1})^2}{2h_0^2e^{-(n-1)}}\right) + \frac12 \\
 & \leq & \cosh^{-1}\left(1+\frac{\frac{n^2}{4}+h_0^2}{2h_0^2e^{-(n-1)}}\right)+\frac12 \\
 & \leq & \cosh^{-1}\left(1+(\frac{n^2}{8}+1)e^{n-1}\right)+\frac12 \\
 & \leq & \cosh^{-1}\left(\frac{\frac{17n^2}{8}e^{n-1}}{2}\right)+\frac12 \\
  & \leq & \cosh^{-1}\left(\cosh(n-1 +2\ln(n) + \ln(17)-\ln(8))\right)+\frac12 \\
 & = & n-1 +2\ln(n) + \ln(17)-\ln(8) +\frac12 \leq n+2\ln(n) 
\end{eqnarray*}
Thus, the volume of the wiring is at most $4\pi\sinh^2((n+2\ln(n)+1)/2)$: the volume of the ball of radius $n+2\ln(n)+1$ in $\HH_\R^2$. We have
\begin{eqnarray*}
 4\pi\sinh^2((n+2\ln(n)+1)/2) & \leq & 4\pi\left(\frac{\exp((n+2\ln(n)+1)/2)}{2}\right)^2 \\ & \leq & \pi\exp(n+2\ln(n)+1) \\ & = & e\pi n^2e^n\simeq e^n
\end{eqnarray*}
as required.
\end{proof}

\section{Questions}
One possible way to improve our bounds on thick embeddings of graphs into other symmetric spaces whose non-compact factor has rank one is via constructions of thick embeddings into nilpotent Lie groups. A positive resolution to the following question would show that the lower bounds from Theorem \ref{thm:wirehypxEuclb} are sharp whenever $Q\geq 2$.

\begin{question}\label{qu:HDKB} Let $P$ be a connected nilpotent Lie group with polynomial growth of degree $p\geq 3$ and let $d\in \N$. Do there exist constants $C,\varepsilon>0$ which depend on $p,d$ such that for any $N$-vertex graph $\Gamma$ with maximal degree $d$ there is a $\varepsilon$-thick embedding of $\Gamma$ into $P$ with diameter $\leq CN^{1/(p-1)}$?
\end{question}

Another important example worthy of consideration is a semidirect product of the Heisenberg group with $\R$, $H\rtimes_{\psi}\R$ where the action is given by
\[
 \left(\begin{array}{ccc} 1 & x & z \\ 0 & 1 & y \\ 0 & 0 & 1 \end{array}\right)\cdot\psi(t) = 
\left(\begin{array}{ccc} 1 & e^tx & z \\ 0 & 1 & e^{-t}y \\ 0 & 0 & 1 \end{array}\right).
\]
\begin{conjecture} For every $d$ there exist constants $C=C(d)$ and $\varepsilon=\varepsilon(d)$ such that every finite graph $\Gamma$ with maximal degree $d$ admits a $\varepsilon$-thick embedding into $H\rtimes_{\psi}\R$ with volume $\leq C|\Gamma|\ln(1+|\Gamma|)$.
\end{conjecture}
An immediate consequence of this conjecture is that the dichotomy at the heart of \cite{HumeMackTess2} is also detected by wiring profiles. Specifically, let $G$ be a connected unimodular Lie group, let $Y$ be a graph quasi-isometric to $G$ and let $X$ be the disjoint union of all finite graphs with degree $\leq 3$. Either $G$ is quasi-isometric to a product of a hyperbolic group and a nilpotent Lie group, in which case there is some $p>1$ such that for all $k$ sufficiently large $\wir^k_{X\to Y}(N) \gtrsim N^p$; or $G$ contains a quasi-isometrically embedded copy of either $\DL(2,2)$ or $H\rtimes_{\psi}\R$, in which case for all $k$ sufficiently large $\wir^k_{X\to Y}(N)\simeq N\ln(N)$. \medskip

The lower bound from separation profiles is incredibly useful, and our best results are all in situations where we can prove that the lower bound in Theorem \ref{thm:wirsep} is optimal. As a result it is natural to record the following:

\begin{question} For which bounded degree graphs $Y$ does the following hold:

Let $X$ be the disjoint union of all finite graphs with maximal degree $\leq 3$. For all $k$ sufficiently large
\[
 \sep_Y(\wir^k_{X\to Y}(N))\simeq N.
\] 
\end{question}
A starting point would be to determine when the following strengthening of Proposition \ref{prop:sepLB} holds:

\begin{question} Let $X,Y$ be graphs of bounded degree where $\wir^k_{X\to Y}(n)<\infty$. Does there exist a constant $C>0$ such that for all $n$
\[
 \sep_Y(\wir^k_{X\to Y}(n)) \geq C^{-1}\sep_X(C^{-1}n)-C?
\]
\end{question}
We certainly should not expect Theorem \ref{thm:wirsep} give the correct lower bound in all cases. A natural example to consider would be a coarse wiring of an infinite binary tree $B$ into $\Z^2$. The depth $k$ binary tree $B_k$ (with vertices considered as binary strings $v=(v_1,v_2,\ldots v_m)$ of length $\leq k$) admits a $1$-wiring into $\Z^2$ with volume $\lesssim |B_k|\ln|B_k|$ as follows
\[
 \psi(v_1,v_2,\ldots v_l) = \left( \sum_{i\in\{l\mid v_l=0\}} 2^{k-i}, \sum_{j\in\{l\mid v_l=1\}} 2^{k-i}\right)
\]
where the path connecting $\psi(v_1,v_2,\ldots v_l)$ to $\psi(v_1,v_2,\ldots v_l,0)$ (respectively $\psi(v_1,v_2,\ldots v_l,1)$) is a horizontal (resp.\ vertical) line.

\begin{question} Is it true that for all sufficiently large $k$, $\wir^k_{B\to\Z^2}(N)\simeq N\ln(N)$? Does the lower bound hold for all coarse wirings $X\to Y$ where $X$ has exponential growth and $Y$ has (at most) polynomial growth?
\end{question}
Since the first version of this paper appeared, this question has been resolved by Kelly \cite{Kelly}, who proved the slightly surprising result that $\wir^1_{T_3\to\Z^2}(N)\simeq N$.

It is also natural to ask whether other invariants which behave monotonically with respect to coarse embedding (and regular maps) provide lower bounds on wiring profiles.

\appendix

\section{Appendix: The Kolmogorov-Barzdin Embedding Theorem in Higher Dimensions}

The goal of this appendix is to prove Theorem \ref{thm:KBRn}. The main theorem roughly says that if we have a graph of bounded degree with $V$ vertices, then we can embed it into an $n$-dimensional Euclidean ball of radius $V^{1/(n-1)}$ without too many edges or vertices intersecting any unit ball. Kolmogorov and Barzdin proved the theorem in dimension 3 and Guth sketched a proof that showed how their method generalized to higher dimensions in the language of thick embeddings. In this appendix we present a full proof using the language of coarse wirings introduced in the present paper.

\begin{theorem} \label{kbthm} \cite{KB},\cite{GQ} Let $Q^1_r$ be the path graph with $r$ vertices, and let
	$$Q^n_r =  Q^1_r \times Q^1_r \times \ldots Q^1_r$$
	where the graph product is taken $n$ times. If $\Gamma$ is a graph where every vertex has degree at most $k$, then for some integer $C>0$ that only depends on $n$ and $k$, and $R = \lceil |V\Gamma|^{\frac{1}{n-1}} \rceil$ there is a coarse $(k+n)$-wiring, 
	
	$$f: \Gamma \to Q^n_{2CR}.$$
\end{theorem}
Here, the graph product $\Gamma_1\times\ldots\times \Gamma_n$ is the graph with vertex set $V\Gamma_1\times \ldots \times V\Gamma_n$ and edges $(v_1,\ldots,v_n),(w_1,\ldots,w_n)$ whenever there is some $j$ such that $v_jw_j\in E\Gamma_j$ and $v_i=w_i$ for all $i\neq j$.

The proof of Theorem \ref{thm:KBRn} follows immediately from Theorem \ref{kbthm} by first applying \ref{thm:thickemb}, then rescaling by a factor of $\varepsilon^{-1}$.

We say a few words about our strategy for constructing $f$. If we think of $Q^n_{2CR}$ as a graph embedded in an $n$-cube, then our $f$ maps the vertices of $\Gamma$ into some face of this cube. The edges of $\Gamma$ are mapped to paths which each consist of $O(n)$ straight segments of randomly chosen lengths. It turns out that this $O(n)$ freedom is enough to guarantee that, with non-zero probability, there is no edge of $Q^n_{2CR}$ where too many of these paths overlap.

In the next section we provide a proof of Theorem \ref{kbthm} following \cite{GQ}.

\section{Proof of Theorem \ref{kbthm}} \label{s2}

\begin{proof} Let $C$ be some large constant, only depending on $n$, which we will choose later. We can think of $Q^n_{2CR}$ as a graph embedded in the cube $[0, 2CR]^n$, where each edge has length $1$. We let $Q^{n-1}_{R}$ be a graph embedded in the bottom face of this cube. Namely, $Q^{n-1}_R$ sits inside $[0, CR]^{n-1} \times 0 \subset [0,2CR]^n$, with each edge having length $C$. Begin by defining $f$ on $V\Gamma$ by embedding all the vertices of $\Gamma$ into $Q^{n-1}_R$ in any way we like. Such an embedding is possible since $R^{n-1}$ is larger than $|V\Gamma|$.
	\par
	Next we have to extend $f$ to the edges of $\Gamma$. Give the edges of $\Gamma$ some order, say $\{e_i\}_{i=1}^{|E\Gamma|}$, and let the endpoints of $e_i$ be $x_{i,-} \in V\Gamma$ and  $x_{i,+} \in V\Gamma$. For many values of $j$, we will construct paths $\gamma(i,j)$ from $f(x_{i,+})$ to $f(x_{i,-})$. Each $\gamma(i,j)$ will consist of $(2n-1)$ straight segments. We select $f(e_1)$ from among the paths $\gamma(1,j)$. Next we select $f(e_2)$ from among the paths $\gamma(2,j)$, making sure it does not pass too close to $\gamma(1,j)$. And so on. At step $i$, we have to see that one of the paths $\gamma(i,j)$ stays far enough away from the previous paths $f(e_1), \ldots f(e_{i-1})$. In fact, we will show that at step $i$, if we choose $j$ randomly, the probability that $\gamma(i,j)$ comes too close to one of the previous paths is less than one half.
	\par
	To define these paths we will use $x_i$ to refer the $i$th coordinate direction in $[0,2CR]^n$. For a set of $n$ integers $j = \{j_0, j_1, \ldots j_{n-1}\} \in ([0, CR] \cap \mathbb{Z})^{n-1}$ the path $\gamma(i,j)$ has the following form. Starting at $f(x_{i,+})$, we first draw a segment in the $x_n$ direction with length $j_0$. Next we draw a segment in the $x_1$ direction with length $j_1$. Then we draw a segment in the $x_2$ direction with length $j_2$. We continue in this way up to a segment in the $x_{n-2}$ direction of length $j_{n-2}$. We have $(CR+1)^{n-1}$ choices for $j$. Then we draw a segment in the $x_{n-1}$ direction which ends at the $x_{n-1}$ coordinate of $f(x_{i,-})$. Then we draw a segment in the $x_{n-2}$ direction which ends at the $x_{n-2}$ coordinate of our target $f(x_{i,-})$, etc. Finally, we draw a segment in the $x_{n}$ direction which ends at $f(x_{i,-})$.
	\par
	We claim that we can choose $j$ so that $\gamma(i,j)$ only intersects previously selected $f(e_i)$ in the $x_n$ direction or intersects them perpendicularly. Since each path is made of segments that point in the coordinate directions, we just have to check that none of the segments intersects a segment of a previous path going in the same direction. Call a segment bad if it intersects a segment from a previous path going in the same direction.
	\par
	The initial segment of $\gamma(i,j)$, in the $x_n$ direction, can intersect at most $k$ segments in the same direction of $f(e_1) \ldots f(e_{i-1})$ because $f$ is an embedding on $V\Gamma$ and $\Gamma$ has degree at most $k$ at each vertex. Consider the first segment in the $x_1$ direction. On this segment, the $2\ldots (n-1)$ coordinates are equal to those of $f(x_{i,+})$. This segment can intersect an $x_1$ segment of a previous path $f(e_{i'})$ only if $f(x_{i,+})$ has the same $2\ldots (n-1)$ coordinates as $f(x_{i',+})$ or $f(x_{i',-})$. This leaves at most $2Rk$ worrisome values of $i'$. But on this segment of $\gamma(i,j)$, the $n$ coordinate is fixed equal to $j_0$. This segment is bad only if it has the same $x_n$ coordinate as a segment in the $x_1$ direction in one of the $2Rk$ worrisome values of $i'$. But there are $CR$ choices for $j_0$. So, choosing $(j_0,\ldots,j){n-1})$ uniformly at random, the probability that this first segment is bad is at most $\frac{2k}{C}$.
	\par
	A similar argument holds for the the second segment. This $x_2$ segment can intersect a previous path $f(e_{i'})$ only if $f(x_{i,+})$ has the same $3 \ldots (n-1)$ coordinates as $f(x_{i',+})$ or $f(x_{i',-})$. This leaves at most $(2Rk)^2$ worrisome values of $i'$ . But there are $(CR)^2$ choices for $(j_0, j_1)$. So the probability that the second segment is bad is at most $(\frac{2k}{C})^2$.
	\par The same reasoning applies for the first $n$ segments. And in fact the same reasoning applies for
	the following $n$ segments as well. For instance, consider the second (and last) segment in the $x_1$ direction. Over the course of this segment, the $2\ldots (n-1)$ coordinates are equal to those of $f(x_{i,-})$, and so this segment can intersect a segment in the $x_1$ direction of a previous path $f(e_{i'})$ only if $f(x_{i,-})$ has the same $2\ldots (n-1)$ coordinates as $f(x_{i',+})$ or $f(x_{i',-})$. This leaves at most $2Rk$ worrisome values of $i'$. But there are more than $CR$ choices of $j_0$, and so the probability that this segment is bad is at most $\frac{2k}{C}$. In summary, there are $2n-1$ segments, and each has probability at most $\frac{2k}{C}$ of being bad, if $C$ is large. So, by a union bound, for some large $C$ only depending on $n$ and $k$, more than half the time, all segments in directions $x_1 \ldots x_{n-1}$ are good. This gives us a path $f(e_i)$ which intersects at most $k$ paths in segments in the $x_n$ direction, and intersects all other previous paths perpendicularly. From the construction, we see that at most $(k+n)$ of the paths we choose intersect any vertex or edge of $Q^n_{2CR}$. In other words, $f$ is a $(k+n)$-coarse wiring.
\end{proof}

\def\cprime{$'$}

\end{document}